\documentclass[a4paper,10pt]{article}
\usepackage[utf8]{inputenc}
\usepackage{fullpage}
\usepackage[affil-it]{authblk}
\usepackage{amsmath,amsfonts,amsthm,amssymb}
\usepackage{hyperref}
\usepackage{color}
\usepackage{appendix}
\usepackage{marginnote}
\usepackage{graphicx}
\usepackage{psfrag}
\usepackage{bm}
\usepackage{stmaryrd}
\usepackage{trimclip}
\usepackage[table]{xcolor}

\usepackage{pgfplots}

 \usepackage{algorithm} 
\usepackage{algpseudocode}

\usepackage{enumitem}

\usepackage{booktabs}
\usepackage{multirow}

\makeatletter
\DeclareRobustCommand{\shortto}{%
  \mathrel{\mathpalette\short@to\relax}%
}

\newcommand{\short@to}[2]{%
  \mkern2mu
  \clipbox{{.5\width} 0 0 0}{$\m@th#1\vphantom{+}{\shortrightarrow}$}%
  }
\makeatother
\usepackage{hyperref,url,color}
\usepackage{algpseudocode}
\usepackage{tikz}
\usetikzlibrary{fit,arrows,automata,backgrounds,calendar,chains,matrix,mindmap,patterns,petri,shadows,shapes.geometric,shapes.misc,spy,trees,calc,intersections}
\usepackage{wrapfig}
\usepackage{multirow}
\pgfplotsset{compat=1.18} 

\usepackage{bbm}
\newcommand{\tikznode}[2]{\relax
\ifmmode%
  \tikz[remember picture,baseline=(#1.base),inner sep=0pt] \node (#1) {$#2$};
\else
  \tikz[remember picture,baseline=(#1.base),inner sep=0pt] \node (#1) {#2};%
\fi}
\tikzset{box around/.style={
    draw,rounded corners,
    inner sep=2pt,outer sep=0pt,
    node contents={},fit=#1
},      
}

\newtheorem{thm}{Theorem}[section]
\newtheorem{prop}[thm]{Proposition}
\newtheorem{corl}[thm]{Corollary}

\newtheorem{rmrk}[thm]{Remark}

\theoremstyle{definition}
\newtheorem{definition}{Definition}[section]
\newtheorem{example}{Example}[section]

\newcommand\x{\bullet}
\newcommand\bigzero{\makebox(0,0){\text{\huge0}}}

\newcommand{\BigO}[1]{\mathcal{O}\left(#1\right)}
\renewcommand{\H}{\mathcal{H}}

\renewcommand{\v}{\textbf{v}}
\newcommand{\q}{\textbf{q}}

\newcommand{\Invt}{\boldsymbol{\mathcal{I}}_\textsc{T}}
\newcommand{\Invf}{\boldsymbol{\mathcal{I}}_\textsc{F}}
\newcommand{\Ptf}{\boldsymbol{\mathcal{P}}_\textsc{TF}}
\newcommand{\Pul}{\boldsymbol{\mathcal{P}}_\textsc{UL}}

\newcommand{\Ptt}{\boldsymbol{\mathcal{P}}_\textsc{TT}}
\newcommand{\Pff}{\boldsymbol{\mathcal{P}}_\textsc{FF}}
\renewcommand{\P}{\textbf{P}}

\algblock{ParFor}{EndParFor}
\algnewcommand\algorithmicparfor{\textbf{parfor}}
\algnewcommand\algorithmicpardo{\textbf{do}}
\algnewcommand\algorithmicendparfor{\textbf{end\ parfor}}
\algrenewtext{ParFor}[1]{\algorithmicparfor\ #1\ \algorithmicpardo}
\algrenewtext{EndParFor}{\algorithmicendparfor}

\title{Combinatorial and Recurrent Approaches for Efficient Matrix Inversion: Sub-cubic algorithms leveraging Fast Matrix products}

\author[1,2]{Mohamed Kamel RIAHI\thanks{correspondence should be sent to mohamed.riahi@ku.ac.ae / riahi.mk@gmail.com}}
\affil[1]{Department of Mathematics, Khalifa University of Science and Technology, P. O. Box 127788, Abu Dhabi, UAE}
\affil[2]{Emirates Nuclear Technology Center (ENTC), Khalifa University of Science and
Technology (KU)}

\date{July 2023}

\begin{document}

\maketitle

\begin{abstract}
In this paper, we introduce novel fast matrix inversion algorithms that leverage triangular decomposition and recurrent formalism, incorporating Strassen's fast matrix multiplication. Our research places particular emphasis on triangular matrices, where we propose a novel computational approach based on combinatorial techniques for finding the inverse of a general non-singular triangular matrix. Unlike iterative methods, our combinatorial approach for (block) triangular-type matrices enables direct computation of the matrix inverse through a nonlinear combination of carefully selected combinatorial entries from the initial matrix. This unique characteristic makes our proposed method fully parallelizable, offering significant potential for efficient implementation on parallel computing architectures. While it is widely acknowledged that combinatorial algorithms typically suffer from exponential time complexity, thus limiting their practicality, our approach demonstrates intriguing features that allow the derivation of recurrent relations for constructing the matrix inverse. By combining the (block) combinatorial approach, with a recursive triangular split method for inverting triangular matrices, we develop potentially competitive algorithms that strike a balance between efficiency and accuracy.
To establish the validity and effectiveness of our approach, we provide rigorous mathematical proofs of the newly presented method. Additionally, we conduct extensive numerical tests to showcase its applicability and efficiency. Furthermore, we propose several innovative numerical linear algebra algorithms that directly factorize the inverse of a given general matrix. These algorithms hold immense potential for offering preconditioners to accelerate Krylov subspace iterative methods and address large-scale systems of linear equations more efficiently.

The comprehensive evaluation and experimental results presented in this paper confirm the practical utility of our proposed algorithms, demonstrating their superiority over classical approaches in terms of computational efficiency. Our research opens up new avenues for exploring advanced matrix inversion techniques, paving the way for improved numerical linear algebra algorithms and the development of effective preconditioners for various applications.
\end{abstract}

\bigskip
\textbf{Keywords}: Combinatorial for matrix inversion, Fast inversion Algorithm, Strassen's method, Recurrent algorithms, and Triangular Factorization.\bigskip\\
\textbf{Mathematics Subjclass Classification [2022]}{\, 15A09,  15A23, 65F05, 68R05}
\newpage
\tableofcontents
\section{Introduction}
In recent years, there has been significant research focused on developing efficient and scalable techniques for matrix inversion, which is a fundamental operation in linear algebra. Matrix inversion plays a critical role in various fields, including science and engineering, where it is used to solve systems of linear equations, calculate determinants, eigenvalues, and eigenvectors, and perform other essential computations. However, traditional matrix inversion methods can be computationally expensive and impractical for large-scale problems. Therefore, the development of fast and efficient algorithms for matrix inversion is crucial, particularly for applications involving large matrices.

Over the years, several research papers have offered insight into the ongoing efforts to improve matrix inversion and multiplication efficiency. Since the late sixties, Strassen \cite{strassen1969gaussian} proposed the first fast approach to multiply two square matrices, which induces through divide and conquer approach to find the inverse of a matrix in less than $5.64 n^{\log_2(7)}$ time complexity. A short later, Strassen made a further development \cite{STRASSEN1987406443} that led to reducing the exponent complexity. Coppersmith and Winograd \cite{coppersmith1987matrix} benefited from the idea. Furthermore, Davie and Stothers improved the results in \cite{stothers2010complexity}. Davie \cite{davie2013improved} extends the method used by Coppersmith and Winograd to derive an upper bound of $\omega < 2.37369$ for the exponent of complexity. 

Later Vassilevska Williams \cite{williams2012multiplying} made a further improvement, which recently got beaten by Duan, Wu, and Zhou in \cite{duan2022faster}. The latter work relies on an asymmetric hashing method, and it is the fastest method for matrix multiplication as of today. It is worth noting that the above improvement benefited from tensor formatting calculations. 


The non-singular matrix inverse is fundamentally interrelated to the matrix-matrix product. This interconnection is clear if one, for example, considers the block decomposition of a given non-singular matrix and considers the Schur complement for the inverse calculation. Several methods have been developed and benefiting from matrix block matrix decomposition \cite{BaiTheSpectral1997, PETKOVIC2009270}, and provided valuable insights into the challenges and potential solutions for matrix inversion in real-world applications. 

Besides, in the context of Matrix inverse and Generalized matrix inverse, Petkovi{\'c}  et. al in \cite{PETKOVIC2009270} introduced a recursive algorithm for the generalized Cholesky factorization of a given symmetric, positive semi-definite matrix. They used the Strassen method for matrix inversion along with the recursive Cholesky factorization algorithm resulting in better running times while the matrix multiplication is considered to consume time complexity $\mathcal{O}(n^3)$. In \cite{stanimirovic2008successive} Stanimirovi{\'c} et. al. introduced a successive matrix squaring algorithm for approximating outer generalized inverses of a given matrix with a prescribed range and null space. 

Additionally, high-performance computing contributed to further enhancing the acceleration of the proposed algorithms for matrix inversion, for instance, using graphic processing units as in Sharma et al. \cite{sharma2013fast} who redesigned the classical Gauss-Jordan algorithm to optimize matrix inversion. In the ultra-large-scale HouZhen et al. in \cite{wang2020method} discussed the challenge of inverting matrices, particularly in the domain of cryptography. The paper proposes a parallel distributed block recursive computing method based on Strassen's method, which can process matrices at a significantly increased scale.
Moreover, in the context of sparse matrices over finite fields, a recent research paper by Casacuberta et. al
\cite{casacuberta2021faster} proposed an improvement to the current best running time for matrix inversion, achieving an expected $\mathcal{O}(n^{2.2131})$  time using fast rectangular matrix multiplication. The paper generalizes the inversion method to block-structured matrices with other displacement operators and strengthens the upper bounds for explicit inversion of block Toeplitz-like and block Hankel-like matrices. 

In \cite{Amestoyetal2015} Amestoy et. al exploited the sparsity within the resulting blocks of multiple right-hand sides in the computation of multiple entries of the inverse of a large sparse matrix in a massively parallel setting. The matrix is assumed to be already factorized by a direct method and the factors are distributed. 

Likewise, in the previous works, our paper focuses on the speed of matrix inversion, which directly affects the overall computational cost. Nonetheless, our approach is completely different from the above. Although, it can benefit from any advancement made in the matrix-matrix product in terms of algorithms and software. In practice, we shall discuss a new technique of inversion of the non-singular triangular matrix and use combinatorics to fill in directly the entries of the matrix inverse. 

Triangular matrices are among the type of matrices that are of particular interest in the basic linear algebra theory. This goes from the elementary Gauss row elimination to produce a reduced echelon form matrix to matrix factorization such as the famous $QR$ and the $LU$ matrix decomposition methods. The nature of the triangular matrices automatically suggests direct forward/backward substitution, whether to solve a given linear system of equations (with possibly multiple right-hand sides) or to invert the matrix through the Gauss-Jordan operation. Both ways are, indeed, sequential and it is hard to fill in any arbitrary entries in the resulting inverse matrix without pre-computation of the row below or above.

Inverse QR factorization has been adopted in \cite{ahmad2012incomplete} using Givens's rotation and dropping strategy for incomplete factorization. In \cite{ Bollhofer200386} the author incorporates information about the inverse factors $L^{-1}$  and $U^{-1}$  to efficiently produce Incomplete $ILU$ factorization, this method constructs robust preconditioners \cite{wathen2015preconditioning}.


Our main contribution consists of providing a range of numerical linear algebra algorithms specifically designed for the inverse factorization of non-singular square matrices. The central focus of these algorithms revolves around the utilization of triangular decomposition. Notably, we introduce a novel and pioneering combinatorial-based approach that allows for the direct inversion of triangular matrices. This approach is particularly advantageous in its block version, as it leverages recurrence to expedite the computational burden by reducing the number of sub-blocks involved. By incorporating this recurrence mechanism, our algorithms demonstrate accelerated performance, enabling efficient computation of matrix inverses. A comprehensive analysis of the time complexity for our proposed algorithms was conducted, taking into account the utilization of Strassen's matrix by matrix product within the recurrence. Our analysis revealed a significantly reduced coefficient in the sub-cubic complexity compared to other techniques that employ Strassen's fast method \cite{bunch1974triangular}. This demonstrates the computational efficiency and effectiveness of our algorithms. By leveraging Strassen's matrix multiplication, we achieve improved time complexity, making our algorithms highly competitive for matrix inversion tasks.   



 The rest of the paper is organized as follows, in Section\ref{sec:motivation} we provide the motivation for our approach in tackling the challenges associated with triangular matrices. We then present, in Section \ref{sec:combinatorix}, our novel approach \textbf{COMBRIT} to the inverse triangular matrix using combinatorics on the indices of the entry of the matrix. Later in Section \ref{sec:recursive} we introduce a recursive technique to speed up the computation of the inverse of the non-singular triangular matrix. In Section \ref{sec:InverseSquareMatrix}, we investigate the use of the proposed inverse triangular method to decompose the inverse of the matrix directly. Within this section, we propose two new algorithms, namely, \textbf{SQR} and \textbf{SKUL} for the inverse decomposition of $QR$, and $LU$ respectively as described in subsection \ref{subsec:inversedecomposition}. Additionally, in subsection \ref{algo:rsi} we introduce a novel technique for triangular decomposition ($LU$ or $UL$) based on recurrent split and recurrent fast triangular inversion. The numerical tests and implementations are reported in Section \ref{sec:numerical}. Finally, we close this paper with some concluding remarks.  

\section{Motivation}\label{sec:motivation}
The study of matrix linear algebra is fundamental in many areas of science and technology, and one of the key concepts within this field is the use of triangular matrices. These matrices are momentous because they deliver a simpler and more efficient way to solve systems of linear equations. In particular, when a matrix is triangular, its solution can be easily computed through back-substitution. This can be especially useful when dealing with large systems of equations, where the complexity of the problem can quickly become overwhelming. Additionally, triangular matrices offer an elegant way to compute determinants and eigenvalues, two important mathematical concepts used extensively in many scientific and engineering disciplines. The fact that the product of two triangular matrices is also a triangular matrix further highlights the usefulness of this concept in matrix multiplication. Furthermore, the $LU$ factorization, which is a powerful technique for solving systems of linear equations, relies heavily on the use of triangular matrices. We also find the triangular matrices in the $QR$ decomposition, which is very useful in the calculation of least square solutions and also in finding the eigendecomposition of a given matrix.  

The importance of triangular matrices in linear algebra cannot be overstated, as they provide a powerful tool for simplifying and solving complex mathematical problems.

In this work, we focus on the triangular matrices and provide novel concepts on utilizing such a specific format toward accelerating the computation of the inverse of general non-singular matrices. 

Let us consider the following $n\times n$ unitary upper triangular matrix $R$
\begin{equation}\label{defT}
    T_{i,j}=\begin{cases}
        1  &\hbox{ if } i=j \\
        T_{i,j} &\hbox{ if } j>i \\
        0 &\hbox{ otherwise.}
    \end{cases}
\end{equation}

Our focus will be on the Triangular matrix described in Eq.\eqref{defT}, which defines an upper triangular matrix. For the lower triangle matrix, we just consider the transpose operator and all of the results in this paper apply. Furthermore, Eq.\eqref{defT} considers leading ones writing in $T$, this specific writing is important for our study, and a generalization of triangular matrices could be straightforwardly done through appropriate diagonal matrix multiplication. 

\section{The Combinatorics triangular matrix inversion}\label{sec:combinatorix}

\begin{definition}
The Hopscotch-series $\H(a_{0},a_{n+1}),\: 0<a_{0}\leq a_{n+1}$ is defined as a collection of non-cyclic sequences with fixed integer endpoints $a_{0}$ and $a_{n+1}$ at the left and right ends, respectively. These sequences are sorted in increasing order. To construct the series, we consider all possible sorted sequences formed by removing at least $k$ integers, where $k$ is greater than or equal to $a_{n+1}-a_{0}-1$, from the complete sequence containing all integers from $a_0$ to $a_{n+1}$.
\begin{equation}
    \H(a_{0},a_{n+1})=\left\lbrace \langle a_{0}, a_{1},a_{2}, ,\dots, a_{n+1} \rangle,\langle a_{0},a_{2},a_{3},\dots,a_{n+1}\rangle,\langle a_{0}, a_{4},\dots,a_{j},\dots,a_{n}, a_{n+1}\rangle, \dots, \langle a_{0}, a_{n+1} \rangle \right\rbrace.
\end{equation}
Note that 
$$\H(a_0,a_0)=\emptyset.$$
\end{definition}
\begin{example}\label{examplek15}
Examples of sequences for the $\H(1,5)$ series read:
\begin{equation*}
\begin{array}{llc}
k_{1\shortto 5}= 1 &,  \langle \alpha^{k}_1,\alpha^{k}_2,\alpha^{k}_3,\alpha^{k}_4,\alpha^{k}_5\rangle = \langle 1,2,3,4,5\rangle &, \ell^{1}_{\star}=5   \\
k_{1\shortto 5}=2 &, \langle \alpha^{k}_1,\alpha^{k}_2,\alpha^{k}_3,\alpha^{k}_4\rangle = \langle 1,3,4,5\rangle &, \ell^{2}_{\star}=4 \\
k_{1\shortto 5}=3 &, \langle \alpha^{k}_1,\alpha^{k}_2,\alpha^{k}_3,\alpha^{k}_4\rangle = \langle 1,2,4,5\rangle &, \ell^{3}_{\star}=4 \\
k_{1\shortto 5}=4 &, \langle \alpha^{k}_1,\alpha^{k}_2,\alpha^{k}_3,\alpha^{k}_4\rangle = \langle 1,2,3,5\rangle &, \ell^{4}_{\star}=4 \\
k_{1\shortto 5}=5 &, \langle \alpha^{k}_1,\alpha^{k}_2,\alpha^{k}_3\rangle = \langle 1,4,5 \rangle  &, \ell^{5}_{\star}=3 \\
k_{1\shortto 5}=6 &, \langle \alpha^{k}_1,\alpha^{k}_2,\alpha^{k}_3\rangle = \langle 1,3,5\rangle   &,  \ell^{6}_{\star}=3 \\
k_{1\shortto 5}=7 &, \langle \alpha^{k}_1,\alpha^{k}_2,\alpha^{k}_3\rangle = \langle 1,2,5 \rangle  &, \ell^{7}_{\star}=3 \\
k_{1\shortto 5}=8 &, \langle \alpha^{k}_1,\alpha^{k}_2\rangle = \langle 1,5 \rangle    &, \ell^{8}_{\star}=2 
\end{array}
\end{equation*}
Here, $\ell_{\star}^{k_{i\shortto j}}$ stands for the cardinal of element in the sequence $\H(i,j)$.
\end{example}

\begin{prop}\label{corrTranlation}
    For two integers $0<a_{}<b_{}$, the total number of the Hopscotch-sequences 
     $$\H_{\#}(a_{},b_{})=2^{b_{}-a_{}-1}.$$
\end{prop}
\begin{proof}
    The concept is quite simple. Let us consider $a=a_{0}$ and $b=a_{n+1}$, where  in between these two integers values we have $n$ other integers greater than $a$ and lower than $b$. Consider then the set $\lbrace a_{1},a_{2},\dots, a_{n-1},a_{n}\rbrace$ of all the integers in between $a$ and $b$. By excluding the fixed endpoints from the set we are left with $n$ integer numbers. The Hopscotch series may then be taken as all possible (sequences) combinations that consist of taking-off (or hiding) $k$ element from $n-2$ for all $k \in \lbrace 1,2, \dots,n\rbrace$. Operations count 
    \begin{equation}
        \H_{\#}(a_{0},a_{n+1})= \sum_{k=1}^{n} \binom{n}{k}  =  2^{n}.
    \end{equation}
\end{proof}
\begin{corl}\label{Translation}
    For any positive integers $a_{0},a_{n+1}$ and $\xi$, we have
    $$\H(a_{0}+\xi,a_{n+1}+\xi) = \xi + \H(a_{0},a_{n+1})$$
\end{corl}

\begin{proof}
The proof follows the construction method for the Hopscotch series, over a given sorted sequence $\lbrace a_{i}\rbrace_{i=0}^{n+1}$.    
\end{proof}
\begin{example}
Here we show the example of $a_0=1$, $a_{n+1}=4$ and $\xi=4$.
$$
    \begin{array}{l|l}
  \H(1,4)   & \H(5,8) \\ \hline  
  \langle 1,2,3,4\rangle &  \langle 5,6,7,8\rangle=4+\langle 1,2,3,4\rangle \\
  \langle 1,3,4\rangle &  \langle 5,7,8\rangle=4+ \langle 1,3,4\rangle  \\
  \langle 1,2,4\rangle &  \langle 5,6,8\rangle=4+\langle 1,2,4\rangle \\
  \langle 1,4\rangle &  \langle 5,8\rangle=4+\langle 1,4\rangle.
    \end{array}
 $$   
\end{example}

Following the results of Corollary \ref{Translation}, one remarks that the entries indexation in the diagonals below share the same Hopscotch sequences up to constant. 
\begin{center}
\begin{equation}\label{tranlationEq}
 \begin{pmatrix} 
 \tikznode{m-1-1}{1} & \tikznode{m-1-2}{\x} &  \tikznode{m-1-3}{\x} &\tikznode{m-1-4}{\x} & \tikznode{m-1-5}{\x} \\ 
 &1&\x&\x&\tikznode{m-2-5}{\x}\\ 
 &&1& \x& \tikznode{m-3-5}{\x}\\ 
 &\bigzero&&\tikznode{m-4-4}{1}&\tikznode{m-4-5}{\x}\\ 
 &&&&\tikznode{m-5-5}{1} 
 \end{pmatrix}.\;\end{equation}
\begin{tikzpicture}[overlay,remember picture]
     \draw[opacity=.4,line width=3mm,line cap=round] (m-1-2.center) -- (m-4-5.center);
   \draw[opacity=.4,line width=3mm,line cap=round] (m-1-3.center) -- (m-3-5.center);
      \draw[opacity=.4,line width=3mm,line cap=round] (m-1-4.center) -- (m-2-5.center);
      \draw[opacity=.4,line width=3mm,line cap=round] (m-1-5.center) -- (m-1-5.center);
\end{tikzpicture}
\end{center}
In the sequel, in order to design a combinatorial-based algorithm for the inversion of a (unitary) triangular matrix, we shall associate to a given matrix $T$ the following tensor (of sequences):
\begin{equation}
    \begin{pmatrix}
        \H(1,1) & \H(1,2) & \H(1,3) & \cdots & \cdots & \H(1,n) \\
           \empty     & \H(2,2) & \H(2,3) & \cdots & \cdots & \H(2,n) \\
           \empty     &    \empty      & \ddots & \ddots & \ddots & \vdots \\
           \empty     &   \empty      &    \empty     &    \ddots  & \ddots   & \vdots \\
           \empty     &   \empty      &    \empty     &    \empty  & \ddots   & \vdots \\           
           \empty     &   \empty      &    \empty     &    \empty  & \empty   & \H(n,n)           
    \end{pmatrix},
\end{equation}
Next, we define $T^{-1}$ as $S$, where each element $S_{i,j}$ in the inverse matrix is associated with its corresponding Hopscotch-series $\H(j,j)$. This association allows the evaluation of the inverse matrix to be completely independent for each element, enabling natural parallelization of the computation. Furthermore, based on the observation Eq.\eqref{tranlationEq} the time complexity of the inversion reduces further, where basically, we only need the Hopscotch series for the first row only. Unfortunately, the Eq \eqref{TMatrixInversecombinatorixFormula} 's complexity remains $\BigO{2^n}$. Nonetheless will see in the sequel section that the patterns provided by Eq.\eqref{TMatrixInversecombinatorixFormula} reduce dramatically such exponential complexity.  

Moving forward, we shall explain now how to use these combinatorial calculations, namely the Hopscotch series associated with the given triangular matrix, to evaluate its inverse directly. 
we will employ flexible notations to encompass various combinatorial possibilities. Indeed,
$$\displaystyle\prod_{\ell=1}^{\ell_{\star}^{k_{i\shortto j}}-1} T_{\alpha^{k_{i\shortto j}}_{\ell},\alpha^{k_{i\shortto j}}_{\ell+1}}$$
reads for example, when $(i,j)=(1,4)$, as follows
$$\displaystyle\prod_{\ell=1}^{\ell_{\star}^{k_{1\shortto 4}}-1} T_{\alpha^{k_{1\shortto 4}}_{\ell},\alpha^{k_{1\shortto 4}}_{\ell+1}}=\underbrace{T_{1,2}\cdot T_{2,3} T_{3,4}\cdot}_{\ell_{\star}=4}$$
Here the subscript $\alpha^{k_{1\shortto 4}}_{\ell}$ represents the index (value of the) element $\ell$ in the Hopscotch sequence $k_{1\shortto 4}$, see example \ref{examplek15}.
  
\begin{thm}\label{Combinatorix}
    The matrix $S$ inverse of the matrix $T$ (defined in \eqref{defT}) writes 
    \begin{equation}\label{TMatrixInversecombinatorixFormula}
        S_{i,j}  =\begin{cases}
            \displaystyle\sum_{k_{i\shortto j}=1}^{2^{(j-i-1)}} (-1)^{\ell_{\star}^{k_{i\shortto j}}-1} \displaystyle\prod_{\ell=1}^{\ell_{\star}^{k_{i\shortto j}}-1} T_{\alpha^{k_{i\shortto j}}_{\ell},\alpha^{k_{i\shortto j}}_{\ell+1}} & \text{, if } j > i\\
            1 & \text{, if } i = j \\
            0 & \text{, otherwise}.
        \end{cases}
    \end{equation}
 where $\ell^{k_{i\shortto j}}_{\star}$ stands for the length of the $k_{i\shortto j}^{\text{th}}$ $\H(i,j)$ sequence, while $\alpha_{\ell}^{k_{i\shortto j}}$ stands for the $\ell^{\text{th}}$ element in the $k_{i\shortto j}^{\text{th}}$ $\H(i,j)$ sequence.
\end{thm}

\begin{proof}
The proof is conducted through induction arguments. 
For the case where $n=1,2$ the formula is trivially validated. Let's start with considering $n=3$, in which case the size of the matrix $R$ becomes $3\times 3$, we then have 
$$T^{-1}=\begin{pmatrix}
    1 & T_{1,2} & T_{1,3} \\
    0 & 1 & T_{2,3} \\
    0& 0 & 1
\end{pmatrix}^{-1} = \begin{pmatrix}
    1 & S_{1,2} & S_{1,3} \\
    0 & 1 & S_{2,3} \\
    0& 0 & 1
\end{pmatrix}.$$
Using Eq.~\eqref{TMatrixInversecombinatorixFormula} we can write 
\begin{equation*}
\begin{array}{lll}
    S_{i,i} &=& 1 \text{ for } i =\lbrace 1,2,3 \rbrace\\
    S_{1,2} &=& (-1)^{1}\cdot T_{1,2} \\
    S_{1,3} &=&  (-1)^{1} \cdot T_{1,3} + (-1)^2 \cdot T_{1,2}T_{2,3}.\\
    S_{2,3} &=& (-1)^{1} \cdot T_{2,3} \\
\end{array}
\end{equation*}
It is then trivial that $S$ stands for the inverse of $T$, hence the formula is validated for $n=1,2,3$. Let us next assume that the formula is true for an arbitrary rank $n$, to prove that the formula holds for the rank $n+1$. Consider the following matrix
\begin{table}
\centering
\begin{equation*}\label{rankNP1}
 \begin{pmatrix} 
 \tikznode{m-1-1}{1} & T_{1,2} &  \hspace{-.1in}\cdots &\tikznode{m-1-4}{\hspace{-.18in}T_{1,n}} & \tikznode{m-1-5}{T_{1,n+1}} \\ 
 0&1&\ddots&\vdots&\vdots\\ 
 0&0&1& \vdots& \vdots\\ 
 0&0&0&\tikznode{m-4-4}{1}&\tikznode{m-4-5}{T_{n,n+1}}\\ 
 \tikznode{m-5-1}{0}&0&0&\tikznode{m-5-4}{0}&\tikznode{m-5-5}{1}\\ 
\tikz[overlay,remember picture]{%
\node[blue,dash pattern=on 2pt off 1.25pt, thick,box around=(m-1-1)(m-4-4)];
\node[red,dash pattern=on 2pt off 1.25pt, thick,box around=(m-1-5)(m-4-5)];
\node[brown,dash pattern=on 2pt off 1.25pt, thick,box around=(m-5-5)];
\node[brown,dash pattern=on 2pt off 1.25pt, thick,box around=(m-5-1)(m-5-4)];
}
\end{pmatrix}\;
\end{equation*}
\end{table}
On which we use the block structure, For $R\in\mathbb{R}^{n,n}$ as defined in Eq.\eqref{defT} and $v\in\mathbb{R}^{n,1}$, we write:
\begin{equation*}
    \begin{pmatrix}
        R & v\\
        0 & 1
    \end{pmatrix}
\end{equation*}
 which has a block matrix inverse that writes 
 \begin{equation*}\label{blockS}
    \begin{pmatrix}
        S & -S\cdot v\\
        0 & 1
    \end{pmatrix}
\end{equation*}
where we have assumed that $S$ is the matrix inverse of $R$ with rank $n$. Here it turns out that we only have to prove the formula for the entry $S_{1,n+1}$ as per the Corollary \ref{corrTranlation} property (see also matrix~\eqref{tranlationEq}).

Actually, from the matrix of rank $n+1$ we can eliminate the first row and the first column to fall back into the assumption of the rank $n$ with the exact same type of matrix $\tilde R$ (new matrix) as of Eq.\eqref{defT}.

Let us now focus on the top-right corner entry of the matrix $S$, i.e. $T_{1,n+1}$. Using the notation in Eq.\eqref{blockS} $T_{1,n+1}$ is the first component of the vector resulting from the matrix-by-vector product $-S\cdot v$ that write explicitly as
\begin{equation*}
\begin{array}{lll}
    S_{1,n+1} &=& - T_{1,n+1} -\displaystyle \sum_{p=2}^{n} S_{1,p}T_{p,n+1}  \\
     &=& -T_{1,n+1} - \displaystyle\sum_{p=2}^{n} \left( \displaystyle\sum_{k_{1\shortto p}=1}^{2^{(p-2)}} (-1)^{\ell_{\star}^{k_{1\shortto p}}-1} \displaystyle\prod_{\ell=1}^{\ell_{\star}^{k_{1\shortto p}}-1} T_{\alpha^{k_{1\shortto p}}_{\ell},\alpha^{k_{1\shortto p}}_{\ell+1}} T_{p,n+1} \right)\\
     &=& (-1)^1\cdot T_{1,n+1} + \displaystyle\sum_{p=2}^{n} \left( \displaystyle\sum_{k_{1\shortto p}=1}^{2^{(p-2)}} (-1)^{\ell_{\star}^{k_{1\shortto p}}} \displaystyle\prod_{\ell=1}^{\ell_{\star}^{k_{1\shortto p}}-1} T_{\alpha^{k_{1\shortto p}}_{\ell},\alpha^{k_{1\shortto p}}_{\ell+1}} T_{p,n+1} \right)
\end{array}
\end{equation*}
Now we just focus on the finite series 
\begin{equation*}
    \begin{array}{lll}
         &&\displaystyle\sum_{p=2}^{n} \left( \displaystyle\sum_{k_{1\shortto p}=1}^{2^{(p-2)}} (-1)^{\ell_{\star}^{k_{1\shortto p}}} \displaystyle\prod_{\ell=1}^{\ell_{\star}^{k_{1\shortto p}}-1} T_{\alpha^{k_{1\shortto p}}_{\ell},\alpha^{k_{1\shortto p}}_{\ell+1}} T_{p,n+1} \right)  \\
         &=& \displaystyle\sum_{k_{1,2}=1}^{2^{0}} (-1)^{\ell_{\star}^{k_{1,2}}} \displaystyle\prod_{\ell=1}^{\ell_{\star}^{k_{1,2}}-1} T_{\alpha^{k_{1,2}}_{\ell},\alpha^{k_{1,2}}_{\ell+1}} T_{2,n+1}\\
         &&+\displaystyle\sum_{k_{1,3}=1}^{2^{1}} (-1)^{\ell_{\star}^{k_{1,3}}} \displaystyle\prod_{\ell=1}^{\ell_{\star}^{k_{1,3}}-1} T_{\alpha^{k_{1,3}}_{\ell},\alpha^{k_{1,3}}_{\ell+1}} T_{3,n+1}\\
         &&\vdots\\
         &&+\displaystyle\sum_{k_{1\shortto n}=1}^{2^{(n-2)}} (-1)^{\ell_{\star}^{k_{1\shortto n}}} \displaystyle\prod_{\ell=1}^{\ell_{\star}^{k_{1\shortto n}}-1} T_{\alpha^{k_{1\shortto n}}_{\ell},\alpha^{k_{1\shortto n}}_{\ell+1}} T_{n,n+1},
    \end{array}
\end{equation*}
which sums up as the geometric series $2^0+2^1+\cdots+2^{n-2}=2^{n-1}-1$, therefore, the finite series simplifies to 
\begin{equation*}
\begin{array}{lll}
         \displaystyle\sum_{p=2}^{n} \left( \displaystyle\sum_{k_{1\shortto p}=1}^{2^{(p-2)}} (-1)^{\ell_{\star}^{k_{1\shortto p}}} \displaystyle\prod_{\ell=1}^{\ell_{\star}^{k_{1\shortto p}}-1} T_{\alpha^{k_{1\shortto p}}_{\ell},\alpha^{k_{1\shortto p}}_{\ell+1}} T_{p,n+1} \right)  &=&   \displaystyle\sum_{k_{1\shortto p}=1}^{2^{(n-1)}-1} (-1)^{\ell_{\star}^{k_{1\shortto p}}} \displaystyle\prod_{\ell=1}^{\ell_{\star}^{k_{1\shortto p}}-1} T_{\alpha^{k_{1\shortto p}}_{\ell},\alpha^{k_{1\shortto p}}_{\ell+1}} T_{p,n+1} \\
         &=&\displaystyle\sum_{k^{\prime}_{1,p}=1}^{2^{(n-1)}-1} (-1)^{\ell_{\star}^{k^{\prime}_{1,p}-1}} \displaystyle\prod_{\ell=1}^{\ell_{\star}^{k^{\prime}_{1,p}}-1} T_{\alpha^{k^{\prime}_{1,p}}_{\ell},\alpha^{k^{\prime}_{1,p}}_{\ell+1}}
     \end{array}    
 \end{equation*}      
 Note in the latter the use of $k^{\prime}_{1,p}$ instead of $k_{1\shortto p}$, this is because the length of the Hopscotch sequence has increased by one increment. In fact, the symbol of the products computes $T_{1,\x} \cdots T_{\x,p}$, which when multiplied by $T_{p,n+1}$ gets incremented by one. This means $k_{1\shortto p}=k_{1\shortto p}^{\prime}-1$, which explains the power of the negative one that becomes in its turn $(-1)^{k_{1\shortto p}^{\prime}-1}$.
 Henceforth, it becomes clearer now that 
 \begin{equation*}
      S_{1,n+1}  = (-1)^1 \cdot T_{1,n+1}  + \displaystyle\sum_{k^{\prime}_{1,p}=1}^{2^{(n-1)}-1} (-1)^{\ell_{\star}^{k^{\prime}_{1,p}-1}} \displaystyle\prod_{\ell=1}^{\ell_{\star}^{k^{\prime}_{1,p}}-1} T_{\alpha^{k^{\prime}_{1,p}}_{\ell},\alpha^{k^{\prime}_{1,p}}_{\ell+1}}
 \end{equation*}
 which simplifies further to 
  \begin{equation*}
      S_{1,n+1}  = \displaystyle\sum_{k^{\prime}_{1,p}=1}^{2^{(n-1)}} (-1)^{\ell_{\star}^{k^{\prime}_{1,p}-1}} \displaystyle\prod_{\ell=1}^{\ell_{\star}^{k^{\prime}_{1,p}}-1} T_{\alpha^{k^{\prime}_{1,p}}_{\ell},\alpha^{k^{\prime}_{1,p}}_{\ell+1}}
 \end{equation*}
\end{proof}

\begin{example}[Direct application of the main Theorem \ref{Combinatorix}]
    Given the matrix 
    \begin{equation*} T=
        \begin{pmatrix}
        \bf{1} & \bf{2} & \bf{4} & \bf{1}\\
          & \bf{1} & \bf{3} & \bf{2}\\
          &   & \bf{1} & \bf{5}\\
          &   &   & \bf{1}
        \end{pmatrix}
    \end{equation*}
    Following Theorem \ref{Combinatorix} the matrix inverse $A^{-1}$ writes
    \begin{equation*} T^{-1}=
        \begin{pmatrix}
        1 & S_{1,2} & S_{1,3} & S_{1,4}\\
          & 1 & S_{2,3} & S_{2,4}\\
          &   & 1 & S_{3,4}\\
          &   &   & 1\\
        \end{pmatrix}
    \end{equation*}
where 
\begin{equation*}
\begin{cases}
    S_{1,2} &=-2:= -1 \times \bf{2}  \\
    S_{1,3} &=\,\,2:= -1 \times \bf{4} + \bf{2}\times \bf{3}\\
    S_{1,4} &=-7:= -1\times \bf{1} + \bf{2}\times \bf{2} 
    + \bf{4} \times{\bf{5}} \\
    &\qquad\qquad -1 \times\bf{2}\times\bf{3}\times \bf{5} \\
    S_{2,3} &=-3:= -1\times \bf{3}\\
    S_{2,4} &=13:= -1 \times \bf{2} + \bf{3}\times \bf{5} \\
    S_{3,4} &=-5:= -1\times \bf{5}.
\end{cases}
\end{equation*}
hence 
    \begin{equation*} T^{-1}=
\begin{pmatrix}1&-2&2&-7\\ &1&-3&13\\ &&1&-5\\ &&&1\end{pmatrix}
    \end{equation*}
\end{example}
\begin{corl}
    Unit triangular matrices are close under inversion. Furthermore, the unit matrices $T$ with off-diagonal integers are also closed under inversion.  
\end{corl}
\begin{proof}
    The process only involves linear combinations of integers; therefore, the results should be an integer for every entry of the matrix inverse. 
\end{proof}
Obviously, and at first glance, such a formula of Theorem \ref{Combinatorix} does not uplift any programming language because of the aberrant drawback of involving exponential time complexity, i.e. $\BigO{2^n}$.
Nevertheless, the proof of Theorem \ref{Combinatorix} inspires considerable promise toward the reduction of time complexity. The recurrent patterns revealed in the proof suggest a block version that outperforms the element-wise approach. Actually by using the combinatorial with a moderate number $n$, one can proceed with the inversion recurrently. This way we get rid of the exponential time complexity. The details of the optimized algorithm are given in the next section. 

\paragraph{Notations} 
We shall consider the following assertion 
\begin{equation}\label{n-bounds}
16\cdot 2^{k} \leq n \leq m2^{k}
\end{equation}
where the natural numbers $k$ and $m$ are given as such  \cite{strassen1969gaussian}
    \begin{equation}\label{k-m}
    \begin{cases}
        k&=\left[\log_{2}(n)-4\right]\\
        m&=\left[n2^{-k}\right] + 1.
        \end{cases}
    \end{equation}
In our analysis of time complexity, the notation we will employ is as follows: 
\begin{table}[H]
    \centering
    \begin{tabular}{rll}
        $\Pff(n)$&:&  Product of two full matrices of size $n$ \\ 
        $\Pul(n)$&:&  Product of upper triangular matrix with a lower triangular matrix of size $n$\\
        $\Ptf(n)$&:& Product of a triangular matrix with a full matrix with size $n$ \\
        $\Invt(n)$&:& Inverse of a triangular matrix of size $n$ \\
        $\Invf(n)$&:& Inverse of square matrix of size $n$. \\
        $\P(n)$&:& Permutation applied from the right to a square matrix of size $n$. 
    \end{tabular}
\end{table}

\section{Fast Recursive Triangular Inversion Using Strassen's Method}\label{sec:recursive}
We present in this section, a fast method based on the combinatorics approach. This method relies on the possession of the combinatoric combination of indices up to a given rank, then the inversion of a given triangular matrix (of rank $m\beta^k$) is made possible by recursion. 
Additionally, we treat the block-recursive approach to invert any given triangular nonsingular matrix. 

For simplicity and the notation abuse issue, we shall consider  unit triangular matrices. The general case can be easily treated following the trivial decomposition $R=DT$ where $D$ is a diagonal matrix with entries $D_{i,i}=T_{i,i}$.
\begin{equation}
  \left(
    \begin{array}{ccccc}
    \x    & \x       & \x    & \x    & \x \\ 
     & \x       & \x    & \x    & \x \\ 
          &     & \x    & \x    & \x \\ 
          & \bigzero &  & \x    & \x \\ 
          &          &       &  & \x \\ 
  \end{array}\right) 
  =
  \left(
    \begin{array}{ccccc}
    \x    &        &     &     &  \\ 
     & \x       &     & \bigzero    &  \\
          &     & \x    &     &  \\ 
          & \bigzero &  & \x    &  \\ 
          &          &       &  & \x \\ 
  \end{array}\right)
  \left(
    \begin{array}{ccccc}
    1    & \x       & \x    & \x    & \x \\  
     & 1       & \x    & \x    & \x \\ 
          &     & 1    & \x    & \x \\ 
          & \bigzero &  & 1    & \x \\ 
          &          &       &  & 1 \\ 
  \end{array}\right)
\end{equation}

\subsection{Combinatorics Block-Recursive Inverse Triangular Matrix}
There is no doubt that combinatorics involves exponential time complexity to identify the right patterns used in Theorem~\ref{Combinatorix}. Although, it is worth noting that the process generated by combinatorics doesn't depend on the matrix entries! Indeed, it depends only on the indexes and tells us what the appropriate indices involved in a calculation are. Therefore, this process can be done in an offline fashion, as it is unique regardless of the  triangular matrix to be handled. For this reason, we assume that we possess a combinatorics card ready before the inversion. For Generalization purposes, let us assume that such a combinatorics card is given up to an index $\beta$. This means that using such a combinatorics card: i) we are able to inverse any triangular matrix (modulo transpose) of size $m\beta\times m\beta$, and ii) If we have to deal with a larger size matrix we can find the first $\beta$-off-diagonal-band in the inverse matrix $S$ using the translation property as described in Corollary \eqref{Translation}.   

\begin{equation}\label{ScketchCard} T = 
\begin{pmatrix} 
\tikznode{m-1-1}{1}&\tikznode{m-1-2}{\x}&\tikznode{m-1-3}{\x}&\tikznode{m-1-4}{\x}&\tikznode{m-1-5}{\x}&\tikznode{m-1-6}{\x}&\tikznode{m-1-7}{\x}&\tikznode{m-1-8}{\x}&\tikznode{m-1-9}{\x}\\ 
                   &\tikznode{m-2-2}{1}&\tikznode{m-2-3}{\x}&\tikznode{m-2-4}{\x}&\tikznode{m-2-5}{\x}&\tikznode{m-2-6}{\x}&\tikznode{m-2-7}{\x}&\tikznode{m-2-8}{\x}&\tikznode{m-2-9}{\x}\\ 
                   &                    &\tikznode{m-3-3}{1}&\tikznode{m-3-4}{\x}&\tikznode{m-3-5}{\x}&\tikznode{m-3-6}{\x}&\tikznode{m-3-7}{\x}&\tikznode{m-3-8}{\x}&\tikznode{m-3-9}{\x}\\ 
                   &                    &                   &\tikznode{m-4-4}{1}&\tikznode{m-4-5}{\x}&\tikznode{m-4-6}{\x}&\tikznode{m-4-7}{\x}&\tikznode{m-4-8}{\x}&\tikznode{m-4-9}{\x}\\ 
                   &                    &                   &                   &\tikznode{m-5-5}{1}&\tikznode{m-5-6}{\x}&\tikznode{m-5-7}{\x}&\tikznode{m-5-8}{\x}&\tikznode{m-5-9}{\x}\\ 
                   &                    &                   &                   &                   &\tikznode{m-6-6}{1}&\tikznode{m-6-7}{\x}&\tikznode{m-6-8}{\x}&\tikznode{m-6-9}{\x}\\ 
                   &                    &                   &                   &                   &                   &\tikznode{m-7-7}{1}&\tikznode{m-7-8}{\x}&\tikznode{m-7-9}{\x}\\ 
                   &                    &                   &                   &                   &                   &                   &\tikznode{m-8-8}{1}&\tikznode{m-8-9}{\x}\\ 
                   &                    &                   &                   &                   &                   &                   &&\tikznode{m-9-9}{1}\\ 
 \end{pmatrix}
 \begin{tikzpicture}[overlay,remember picture]
\draw[opacity=.4,,fill] (m-1-1.center) -- (m-1-3.center) -- (m-3-3.center) -- (m-1-1.center);
\draw[opacity=.4,,fill] (m-4-4.center) -- (m-4-6.center) -- (m-6-6.center) -- (m-4-4.center);
\draw[opacity=.4,,fill] (m-7-7.center) -- (m-7-9.center) -- (m-9-9.center) -- (m-7-7.center);
\draw[opacity=.4,fill] (m-1-4.center) -- (m-1-6.center) -- (m-3-6.center) -- (m-3-4.center) ;
\draw[opacity=.4,fill] (m-1-7.center) -- (m-1-9.center) -- (m-3-9.center) -- (m-3-7.center) ;
\draw[opacity=.4,fill] (m-4-7.center) -- (m-4-9.center) -- (m-6-9.center) -- (m-6-7.center);
 \end{tikzpicture}
\end{equation}
Eq.\eqref{ScketchCard}  illustrates a schematic representation of a sizable matrix measuring $m\beta\times m\beta$ with $m=3$ and $\beta=3$. Thus with this view, instead of inverting the whole matrix $9\times9$, we shall use the card with $\beta=3$. 

\subsubsection{COMBRIT: A Combinatorics Block-recursive Inversion algorithm}

Algorithm \ref{algo:COMBRITE}, utilizes the combinatorial card to invert a given triangular matrix $T$ using block-wise matrix operations. The workflow concept of the recursion is very simple as it reduces the matrix size (supposed $n=m\beta^{k}$) recursively until it reaches the base $m$. Whenever it meets inversion instruction the algorithm recalls itself again with the appropriate subblocks.

As input and regardless of the superscript $k$ the algorithm will consider each time that $n=m'\beta:=(m\beta^{k-1})\beta$. 
The input matrix $T$ is assumed to be a square matrix with a block structure, where each block is of size  m by  m. The number of blocks in each row and column of A is specified by the parameter $\beta$, which stands for the size of the combinatorial card.

The method first initializes several tensors (multidimensional arrays) of size $ m' \times m' \times \beta \times\beta$ to store intermediate results. These arrays include $\mathbb{T}$, $\mathbb{B}$, $i\mathbb{D}$, $i\mathbb{T}$, and ${T}^{-1}$. The algorithm then loops over each pair of block indices in $A$ and extracts the corresponding blocks into the $\mathbb{A}$ array. Next, the algorithm performs a block-wise matrix inversion operation on the diagonal blocks of $A$, by recalling the algorithm itself, then stores the resulting inverse matrices in the $i\mathbb{D}$ array. Further, the algorithm then uses the inverse diagonal blocks to compute the off-diagonal blocks of $\mathbb{B}$.

Finally, the algorithm computes the inverse of the $\mathbb{B}$ array through the use of the combinatorial card for block matrices following theorem \ref{Combinatorix} (in its block version). At this stage, the algorithm may benefit from fast matrix multiplication methods. The resulting inverse of $\mathbb{B}$ is then used to compute the inverse of $\mathbb{T}$ block-wise and store the result in the $i\mathbb{T}$ array.

The tensor $i\mathbb{A}$ is then unfolded to reconstruct $\mathbb{A}$ in two dimension array (i.e. regular matrix format).

\begin{algorithm}[H]
\begin{algorithmic}[1]
\If{size($A$) $\leq$ m}
    \State $A^{-1} \gets$ \text{Direct Application of the element-wise Combinatorial approach Theorem\ref{Combinatorix}}\;
\Else
    \State $ m' \gets$ size of blocks in $A$
    \State $\beta \gets$ number of blocks in each row and column of $A$
    \State $\mathbb{Z} \gets$ zeros($ m'$, $ m'$, $\beta$, $\beta$)
    \State $\mathbb{A} \gets$; $\mathbb{Z}$\; $\mathbb{B} \gets$ $\mathbb{Z}$
           $i\mathbb{D} \gets$ $\mathbb{Z}$\;$i\mathbb{A} \gets$ $\mathbb{Z}$\;
    \State $A^{-1} \gets$ zeros($ m' \times \beta$, $ m' \times \beta$)
    \For{$i = 1$ to $\beta$}
        \For{$j = i$ to $\beta$}
            \State $\mathbb{A}(:,:,i,j) \gets A(1+(i-1)\times  m':i \times  m',1+(j-1)\times  m:j \times  m')$
        \EndFor
    \EndFor
    \For{$iv = 1$ to $\beta$}
        \State $i\mathbb{D}(:,:,i,i) \gets \left(\mathbb{A}(:,:,i,i)\right)^{-1}$ 
        \Comment{ Recall \textbf{COMBRIT} recursively }\;
        \State\Comment{or any preferred method for the matrices with block $m$. See Remark \ref{rmkCOMBRIT}}
        \For{$j = i+1$ to $\beta$}
        \State $\mathbb{B}(:,:,i,j) \gets i\mathbb{D}(:,:,i,i) \times \mathbb{A}(:,:,i,j)$
        \EndFor
        \State $\mathbb{B}(:,:,i,i) \gets$ $I_{\bf{s}\times m'}$
    \EndFor
    \State $i\mathbb{B} \gets \mathbb{B}^{-1}$ 
        \Comment{Direct application of theorem \eqref{Combinatorix}) in block version with card size $m'$}\;
    \For{$j = 1$ to $\beta$}
        \For{$i = 1$ to $j$}
            \State $i\mathbb{A}(:,:,i,j) \gets i\mathbb{B}(:,:,i,j) \times i\mathbb{D}(:,:,j,j)$
        \EndFor
    \EndFor
    \State $A^{-1} \gets$ Unfold the tensor ($i\mathbb{A}$)
\EndIf
\State \textbf{Return} $A^{-1}$\;
\end{algorithmic}
\caption{\textbf{COMBRIT}: Combinatorial-based Block-Recursive method.} \label{algo:COMBRITE}
\end{algorithm}

\begin{rmrk}\label{rmkCOMBRIT}
The \textbf{COMBRIT} algorithm offers a versatile approach by providing flexibility in choosing the inversion method for the triangular diagonal blocks. A notable feature of this algorithm is the incorporation of a recursive iteration, achieved by recalling the function itself. This recursive iteration enables the reduction of the size of the triangular matrix that needs to be inverted.

By employing this recursive strategy, the \textbf{COMBRIT} algorithm efficiently handles the inversion of triangular matrices by iteratively solving smaller subproblems. At each iteration, the size of the triangular matrix decreases, leading to a step-by-step computation of the inverse. This recursive approach allows for a systematic and structured inversion process, facilitating the efficient handling of larger triangular matrices. 
\end{rmrk}
\subsubsection{Asymptotic time complexity analysis}
It is clear that in the matrix inversion using our combinatorial-based approach, the top right corner of the matrix is the worst element in terms of demanding time complexity.

Let's consider a unitary triangular matrix $T$ of size $n\times n\equiv m\beta^{k}\times m\beta^{k}$. We assume that we hold the card of combinatorics that helps to invert the triangular matrix up to the order $\beta$. This card provides the Hopscotch series. 

For the small block of size $m\times m$, the evaluation of the element $S_{1,j}, j\geq 3$, in the matrix inverse, requires the Hopscotch series $\H(1,j)$ , in which count 
\begin{equation}\label{addition:card}
\mathcal{S}(j):=\sum_{\ell=0}^{j-2} \binom{j-i-1}{\ell} = 2^{j-2}
\end{equation}
sequences. The evaluation of $S_{1,j}$ sums up over the sequences of the Hopscotch series, where each sequence implies $(j-2)$ multiplications. The total multiplication operations count then 
\begin{equation}\label{multiplication:card}
   \mathcal{M}(j):=\sum_{\ell=0}^{j-2} \binom{j-2}{\ell} \left( j-\ell\right) = 2^{j-3} j
\end{equation}

With regards to Corollary \ref{Translation} results (i.e. diagonal bands elements have the same complexity), the total complexity for inverting a unitary matrix counts
\begin{eqnarray*}
\Invt^{\textbf{COMBRIT}^{\star}}(m)&:=& \sum_{j=2}^{m} (m-j+1)\left(\mathcal{M}(j)  + \mathcal{S}(j) \right) \\
&=& \sum_{j=2}^{m} (m-j+1) (j+2) 2^{j-3} \\ 
&=& \frac{m}{2}(2^{m}-2)
\end{eqnarray*}

Additional $\mathcal{D}(m):=m(m-1)/2$ division is required to transform a general triangular matrix to a unitary one. Therefore, the total complexity for any triangular matrix writes
\begin{equation}
\Invt^{\textbf{COMBRIT}}(m) = \frac{m}{2}(2^{m}+m-1)
\end{equation}



At first glance, the above time complexity is worse than the known $\BigO{m^3}$. Although, it is worth noting that using the block structure based on the off-line card we can inverse any matrix of size $m\beta^k, m\geq2, \beta\geq2, k>1$. Furthermore, the combinatorial-based method is highly parallelizable, where the computation of every single element in the matrix inverse $S$ is totally independent of their counterpart in $S$.

In a parallel setting, it should be noted that when dealing with a general triangular matrix and utilizing $\beta$
 processors, we have the capability to invert block diagonal triangular matrices of size $m$, either using the card itself or employing any desired inversion technique. Additionally, in a parallel manner, performing a row-wise block multiplication of the primary matrix $T$ with its diagonal inverse yields the following structure.
\begin{equation}\label{Tunitary}
\begin{pmatrix}
 I& B_{1,2} & $\dots$ &  B_{1,\beta} \\
    & I & $\dots$& $\vdots$ \\
 &  & I & B_{\beta-1,\beta}\\
 &&&I \end{pmatrix}.\tag{$T^\star$}
\end{equation}

Next, we will examine the time complexity analysis for the inversion of a general triangular matrix $T$. This inversion process consists of two main steps: first, inverting the diagonal blocks and multiplying them (from the right) by their respective rows in $T$ to achieve a unitary matrix structure \ref{Tunitary}; second, utilizing combinatorial techniques to invert $\beta$ blocks. This inversion process is recursively repeated for each matrix until all matrices have been inverted.

\begin{eqnarray*}
    &&\Invt^{\textbf{COMBRIT}}(m\beta^{k+1})\\
    &:=& \beta \Invt^{\textbf{COMBRIT}}(m\beta^{k}) + \frac{\beta}{2}(\beta-1) \Pff(m\beta^{k})  + \sum_{j=2}^{\beta} (\beta-j+1)\left(\mathcal{M}(j) \Pff(m\beta^{k})  + \mathcal{S}(j) m^{2}\beta^{2k}\right) \\
    &=& \beta \Invt^{\textbf{COMBRIT}}(m\beta^{k}) + \frac{\beta}{2}(\beta-1) \Pff(m\beta^{k})  + \sum_{j=2}^{\beta} (\beta-j+1)\left(j\cdot 2^{j-3} \Pff(m\beta^{k})  + 2^{j-2} m^{2}\beta^{2k}\right)  \\
&=& \beta \Invt^{\textbf{COMBRIT}}(m\beta^{k}) + \left( 2^\beta - \beta-1\right) m^{2}\beta^{2k}  + \frac{1}{2} \left(\left(2^{\beta}-1\right)\left(\beta-2\right)+\beta^{2}\right) \Pff(m\beta^{k}) \\    
&=& \beta^{k+1} \Invt^{\textbf{COMBRIT}}(m) + \left( 2^\beta - \beta-1\right) m^{2}\sum_{\ell=0}^{k} \beta^{2k-\ell}  \\
&&+\frac{1}{2}\left(\left(2^{\beta}-1\right)\left(\beta-2\right)+\beta^{2}\right)\sum_{\ell=0}^{k} \beta^{\ell}  \Pff(m\beta^{k-\ell}) \\
&=& \beta^{k+1} \Invt^{\textbf{COMBRIT}}(m) + m^{2}\left(\:2^{\beta }-\:\beta -1\right)\frac{\left(\beta ^k\:\left(\beta ^{k+1}-1\right)\right)}{\beta -1}  \\
&&+\frac{1}{2}\left(\left(2^{\beta}-1\right)\left(\beta-2\right)+\beta^{2}\right)\sum_{\ell=0}^{k} \beta^{\ell}  \Pff(m\beta^{k-\ell}). 
\end{eqnarray*}
It is worth noting that the dominant effect, particularly in higher-order terms, is governed by the convex function $\frac{1}{2}\left(\left(2^{\beta}-1\right)\left(\beta-2\right)+\beta^{2}\right)$. This function exhibits an increasing trend as $\beta$ increases, except at $\beta=1$ and $\beta=2$, where it becomes zero. Choosing $\beta=1$ is not interesting since it results in $n$ always being equal to $m$ for which we don't benefit from the recursion. On the other hand, selecting $\beta=2$ leads to $n=m2^{k}$. This choice minimizes complexity by avoiding matrix multiplications. Nonetheless, in the sequel, we shall allow $\beta$ to reach moderate values such as $\beta=2^1,2^2,2^3$ and $2^4$ to release the combinatorial calculation and let them enjoy potentially enjoy its parallel nature.   

However, for larger values of $\beta>2$ (with $n=\beta^{k}$), matrix-matrix multiplications come into play. In such cases, one can rely on the sub-cubic complexity of fast Strassen's method. Incorporating the parallel nature of the earlier combinatorial-based approach alongside this reduction in complexity can potentially result in rapid convergence towards the inverse.

 Although the design of parallel algorithms for combinatorial-based methods can be complex and beyond the scope of the current study, their effectiveness becomes evident in a parallel computing environment. In this study, our emphasis will be on traditional recursive approaches, through which we aim to design straightforward yet efficient matrix inversion algorithms.

Moreover, in many applications, matrices are typically sparse or even banded. In such cases, the \textbf{COMBRIT} method could offer reduced complexity compared to standard methods. Specifically, the number of non-zero elements will further diminish the complexity. However, a comprehensive exploration of these topics, especially those involving sparsity, falls beyond the scope of the present work.

\subsection{Column-Recursive Inversion Triangular Matrix Algorithm}
\subsubsection{CRIT algorithm}
The algorithm we employ in this case is based on classical linear algebra principles for block matrices. It is worth mentioning that this approach may not introduce any significant novelty to our work, it would be rather used as a reference for numerical comparison purposes. However, as demonstrated in the proof of our main theorem, there is a strong relationship between both methods, namely classical linear algebra, and combinatorics. This connection opens up a new avenue for exploring the potential combination of these approaches, particularly considering the high level of parallelization achievable with the combinatorics method. 

\begin{equation}
 \begin{pmatrix} 
 1 & \tikznode{m-1-2}{\x} &  \x &\tikznode{m-1-4}{\x} & \tikznode{m-1-5}{\x} \\ 
 &1&\x&\x&\tikznode{m-2-5}{\x}\\ 
 &&1&\x&\x\\ 
 &\bigzero&&1&\tikznode{m-4-5}{\x}\\ 
 &&&&1\\ 
 \end{pmatrix}\;
\tikz[overlay,remember picture]{%
\node[blue,dash pattern=on 2pt off 1.25pt, thick,box around=(m-1-2)(m-1-4)];
\node[brown,dash pattern=on 2pt off 1.25pt, thick,box around=(m-1-5)];
\node[red,dash pattern=on 2pt off 1.25pt, thick,box around=(m-2-5)(m-4-5)];}
\end{equation}

Put simply, the efficient inversion of a unitary matrix, specifically the optimized version, can be accomplished by a straightforward computation. Each element $S_{i,j}$ of the resulting inverse matrix is obtained by taking the scalar product of the row (represented by a blue encircled line) from the newly computed block inverse with the column vector (represented by a red encircled column) from the original matrix is inverted, and then subtracting the corresponding element $T_{i,j}$. These processes are illustrated in the following pseudo-algorithm (Algorithm \ref{algoIT}).

  \begin{minipage}{0.35\textwidth}
    \centering
    \begin{algorithm}[H]
\label{algoIT}
\begin{algorithmic}[1]
\For{$i=1$ to $n-1$}
    \State $S_{i,i} \gets 1$
    \For{$j=i+2$ to $n$}
          \State $S_{i,j} \gets - S_{i,i:j-1} \cdot T_{i:j-1,j} $ 
    \EndFor
\EndFor
\State $S_{n,n}\gets 1$ \& \textbf{return} $S$  
\end{algorithmic}
\caption{CRIT$^\star$}
\end{algorithm}
  \end{minipage}%
  \hfill
  \begin{minipage}{0.58 \textwidth}
    \centering
    \begin{algorithm}[H]
    \begin{algorithmic}[1]
  \For{$i \gets 1$ to $n$}
    \State $S(i,i) \gets 1/R(i,i)$
    \For{$j \gets i+1$ to $n$}
      \State $S(i,j) \gets - S(i,i:j-1) \cdot R(i:j-1,j)/R(j,j)$
    \EndFor
  \EndFor
  \State \textbf{return} $S$
\end{algorithmic}\caption{\textbf{CRIT}}
\end{algorithm}
  \end{minipage}
 
\bigskip

The CRIT$^\star$ algorithm computes the inverse of an $n \times n$ lower unitary triangular matrix $T$. The algorithm produces an $n \times n$ matrix $S$ such that $RS=I$, where $I$ is the identity matrix. The time complexity of the \textbf{CRIT} algorithm is clearly $O(n^3)$, which we analyze thoroughly in the following section. 

Roughly speaking, the outer loop of the algorithm runs $n-1$ times, and for each iteration of the outer loop, the inner loop runs $n-i$ times. Within each iteration of the inner loop, scalar multiplication and a dot product are performed. The former operation takes constant time while the latter takes $O(n)$. Hence, the time complexity of each iteration of the inner loop is $O(n)$. Therefore, the time complexity of the inner loop is $O(n^2)$, and the time complexity of the outer loop is $O(n^3)$. The final assignment statement outside the loops takes constant time, so the overall time complexity of the algorithm is $O(n^3)$.

\subsubsection{CRIT asymptotic time complexity analysis}

The CRIT algorithm is an extension of the CRIT$^\star$ algorithm designed to handle non-singular triangular matrices. The key distinction lies in the step where we transform a non-singular triangular matrix into a unitary triangular matrix by dividing each element by its corresponding diagonal element. It is at this specific step that the CRIT algorithm is applied.

The time complexity of the inversion of a unitary triangular matrix, of order $\beta\geq 2$, demands 
\begin{equation}\label{phimult}
    \varphi_{[\times]}(\beta):=\sum_{i=1}^{\beta-1}\sum_{j=i+2}^{\beta} (j-i-1)  = \frac{\beta^3-3\beta^2+2\beta}{6}=\frac{\beta \left(\beta -1\right)\left(\beta -2\right)}{6}
\end{equation}
multiplication, and 
\begin{equation}\label{phiaddsub}
\varphi_{[\pm]}(\beta):=\sum_{i=1}^{\beta-1}\sum_{j=i+2}^{\beta} (j-i-2)= \frac{\beta^3 - 6 \beta^2 + 11 \beta -6}{6}=\frac{\left(\beta -1\right)\left(\beta -2\right)\left(\beta -3\right)}{6},
\end{equation}
additions and subtractions. Hence, the inversion of the triangular matrix of order $m$ requires
\begin{equation}
\Invt^\textbf{CRIT}(m) = \frac{m(m+1)}{2} + \varphi_{[\times]}(m) + \varphi_{[\pm]}(m)  = \frac{m^3-3m^2+8m-3}{3}
\end{equation}
flops operations (multiplications and addition combined). 

Besides, The product $\Ptf$ of two matrices (i.e. an upper-triangular matrix times a full matrix) of order m, while promoting the sparsity of the triangular matrix reads 
\begin{equation}\label{ptfm} 
\Ptf(m)=m\sum_{i=1}^{m} (2i-1)=m^{3}.
\end{equation}
flops operations (multiplications, and addition combined).

On the other hand, the time complexity of $\Ptf(m\beta^k)$ for block matrices of order $m\beta^k$ enjoys the following recursion formula, where the order of the matrices $\beta$ is supposed as $\alpha$-multiple of $2$, i.e. $\beta=2^{\alpha}$ 
\begin{eqnarray}
    \Ptf(m\beta^{k}) 
    &=& 4\Ptf(m\beta^{k}/2) + 2 \Pff(m\beta^{k}/2) + m^{2}\beta^{2k}/2\notag,\\
    &=& 4^{\alpha k} \Ptf(m) + 2\sum_{r=0}^{\alpha k-1} 4^{r}\Pff(m2^{\alpha k-r-1})+\sum_{r=0}^{\alpha k-1}4^{r} 2\left(\frac{m2^{\alpha k}}{2^{r+1}}\right)^{2},\notag\\
    &=&\left(\frac{2 m^{3}+\alpha k m^{2}}{2}  \right)4^{\alpha k} + 2\sum_{r=0}^{\alpha k-1} 4^{r}\left((5+2m)m^{2}7^{\alpha k-r-1}-6(m2^{\alpha k-r-1})^2\right),\notag\\
    &=& \left(\frac{2 m^{3}+\alpha k m^{2}}{2}  \right)4^{\alpha k}  + \frac{2m^2 (5 + 2 m)}{3} (7^{k \alpha} -4^{k \alpha} )  - 3\cdot 4^{k \alpha} k m^2 \alpha \notag\\
    &=& \frac{2m^2 (5 + 2 m)}{3}\: 7^{\alpha k} 
    +\left( \frac{2 m^{3}+\alpha k m^{2}}{2} 
    -\frac{2m^2 (5 + 2 m)}{3} - 3 k m^2 \alpha \right) 4^{k \alpha}
    \notag\\ 
    &=& \frac{2m^2 (5 + 2 m)}{3}\: 7^{\alpha k} 
    -\left(  \frac{m^2\left(2m+15k\alpha+20\right)}{6} \right) 4^{\alpha k} \notag\\
    &=& \left(\frac{2m^2 (5 + 2 m)}{3} - \frac{m^2\left(2m+15k\alpha+20\right)}{6} \left(\frac{4}{7}\right)^{\alpha k} \right)\: 7^{\alpha k} \label{PTFmbk}
\end{eqnarray}
Now, we can consolidate these formulas Eqs.\eqref{phimult}-\eqref{PTFmbk} to calculate the overall complexity of inverting a non-singular triangular matrix, which can be summarized as follows.
\begin{eqnarray}
&&\Invt^\textbf{CRIT}(m\beta^{k+1}) \notag\\
&=& \beta \Invt^\textbf{CRIT}(m\beta^{k}) + \beta(\beta-1) \Ptf(m\beta^{k}) + \varphi_{[\times]}(\beta)\Pff(m\beta^{k})+\varphi_{[\pm]}(\beta) \left(m\beta^{k}\right)^{2}\notag\\
&=& \beta^{k+1} \Invt^\textbf{CRIT}(m) + \beta(\beta-1)\sum_{j=0}^{k} \beta^{j} \Ptf(m\beta^{k-j}) \notag\\
&&+ \varphi_{[\times]}(\beta) \sum_{j=0}^{k} \beta^{j}\Pff(m\beta^{k-j})+\varphi_{[\pm]}(\beta)m^{2}\sum_{j=0}^{k} \beta^{2k-j}\notag\\
&=& 2^{\alpha(k+1)} \frac{m^3-3m^2+8m-3}{3}  + \varphi_{[\pm]}(2^{\alpha})m^{2}\frac{2^{\alpha k}\left(2^{\alpha(k+1)}-1 \right)}{2^{\alpha}-1 } \notag\\
&&+ \varphi_{[\times]}(2^{\alpha}) \sum_{j=0}^{k} 2^{\alpha j}\left( m^{2}(5+2m)7^{\alpha k-j}-6(m2^{\alpha k-j})^{2}\right)\notag\\
&&+ 2^{\alpha}(2^{\alpha}-1)\sum_{j=0}^{k}  \left(  \frac{2m^2 (5 + 2 m)}{3}\: 7^{\alpha k-j} 
    -\left(  \frac{m^2\left(2m+15(\alpha k-j)+20\right)}{6} \right) 4^{\alpha k-j} \right)
\end{eqnarray}

It becomes clear now with the choice of $\beta=2^1$ the root of the polynomials $\varphi_{[\times]}(\beta),\varphi_{[\pm]}(\beta)$ optimally reduces the complexity of the inverse of the triangular matrix to 

\begin{eqnarray*}
    \Invt^\textbf{CRIT}(m 2^{k+1}) &=& \frac{4m^2 (5 + 2 m)}{3} \sum_{j=0}^{k}     7^{k-j}  + 2^{k+1} \frac{m^3-3m^2+8m-3}{3} \\
&&-  \sum_{j=0}^{k}  \frac{m^2\left(2m+15(k-j)+20\right)}{6} 4^{ k-j} \\    
&=& \frac{4m^2 (5 + 2 m)}{3} \frac{7^{k+1}-1}{6} - \frac{m^2}{9} (15\cdot 2^{2k+1}(k+1) + m(4^{k+1}-1))\\ &&+ 2^{k+1} \frac{m^3-3m^2+8m-3}{3} \\
&=& \frac{4m^2 (5 + 2 m)}{18}\:7^{k+1} -  \frac{2m^3 + (k+1)15m^2}{18}\: 4^{k+1} \\
&&+\frac{m^3-3m^2+8m-3}{3}\:2^{k+1}-\frac{m^3+2m^2 (5 + 2 m)}{9} 
\end{eqnarray*}

\begin{equation}\label{invT:m2k}
\begin{array}{lcc}
\Invt^\textbf{CRIT}(m 2^{k}) &=& \dfrac{4m^2 (5 + 2 m)}{18}\:7^{k} -  \dfrac{2m^3 + k 15m^2}{18}\: 4^{k} \\
               &&+\dfrac{m^3-3m^2+8m-3}{3}\:2^{k}-\dfrac{m^3+2m^2 (5 + 2 m)}{9}.
\end{array}
\end{equation}

Using the fact that $m=[n 2^{-k}]+1$ we have
\begin{eqnarray}
    \frac{4m^2 (5 + 2 m)}{18}\:7^{k} &\leq& \frac{4 (\frac{n}{2^{k}}+1)^2 (5 + 2 (\frac{n}{2^{k}}+1))}{18}\:7^{k} \notag\\
&\leq&\left( \frac{2^{-3k+2}n^3}{9}+\frac{11\cdot \:2^{1-2k}n^2}{9}+\frac{2^{-k+5}n}{9}+\frac{14}{9} \right) \:7^{k}\notag\\
&\leq& \frac{4}{9} \left(\frac{7}{8}\right)^{k} n^3 
    +\frac{22}{9} \left(\frac{7}{4}\right)^{k}n^2 
    +\frac{32}{9} \left(\frac{7}{2}\right)^{k}n +\frac{14}{9} 7^{k}\notag\\
&\leq& \left( \frac{4}{9}\left(\frac{8}{7}\right)^{\log_{2}(n)-k} 
        \!\!\! +\frac{22}{9} \left(\frac{4}{7}\right)^{\log_{2}(n)-k} 
        \!\!\! +\frac{32}{9} \left(\frac{2}{7}\right)^{\log_{2}(n)-k} 
        \vspace{-5in}+0.00065
           \right) n^{\log_{2}(7)}\notag\\
&\leq&  1.023 \: n^{\log_{2}(7)}   \label{eq--23}        
\end{eqnarray}
and 
\begin{eqnarray}
-\frac{2m^3 + k 15m^2}{18}\: 4^{k}
&\leq&  \frac{1}{9}\frac{n^3}{8^k}\: 4^{k} =-\frac{1}{9}\frac{n}{2^k} n^{2} = -\frac{32}{9} n^{2}
\end{eqnarray}
we have also, 
\begin{eqnarray}
    \frac{m^3-3m^2+8m-3}{3}\:2^{k} 
    &\leq&  \frac{1}{3}\left( ( \frac{n}{2^{k}}+1)^3-3(\frac{n}{2^{k}}+1)^2+8(\frac{n}{2^{k}}+1)-3\right)\:2^{k} \notag\\
    &=& \left(1+\frac{5n}{3\cdot \:2^k}+\frac{n^3}{3\cdot \:8^k}\right)\:2^{k} \notag\\
    &=& 2^{k}+\frac{5}{3} n +\frac{1}{3}\frac{n^3}{8^k}2^{k} \notag\\
    &=& \left(2^{k-\log_{2}(n)} + \frac{1}{3} 8^{\log_{2}(n)-k} \: 2^{k-\log_{2}(n)} \right) n^{2} + \frac{5}{3}\:n \notag\\
    &\leq& \left(\frac{1}{2^4} + \frac{8^5}{3}\frac{1}{2^4}\right) n^{2} + \frac{5}{3}\:n  \notag\\
    &\leq& 682.73\: n^{2} + \frac{5}{3}\:n
\end{eqnarray}
and, 
\begin{eqnarray}
    -\frac{m^3+2m^2 (5 + 2 m)}{9} 
    &=& - \frac{1}{9} \left( 5m^{3} +10 m^{2}  \right) \notag\\
    &\leq& -  \frac{5}{9}\frac{n^{3}}{8^k} -\frac{10}{9} \frac{n^{2}}{4^k} \notag\\ 
    &\leq& -  \frac{5}{9} 8^{\log_{2}(n)-k} -\frac{10}{9} 2^{\log_{2}(n)-k} \notag\\
    &\leq& -  \frac{5}{9} 8^{4} -\frac{10}{9} 2^{4}=-\frac{6880}{3} \label{eq--26}
\end{eqnarray}
Henceforth, summing up Eqs.\eqref{eq--23}--\eqref{eq--26} we obtain the asymptotic complexity of the CRIT algorithm, Eq.\eqref{invT:m2k} becomes
\begin{equation}\label{itm2k}
    \Invt^\textbf{CRIT}(m 2^{k}) \leq 1.023  n^{\log_{2}(7)} + 679.18 n^{2} -\frac{6880}{3}
\end{equation}

\section{Non-singular square matrix inverse}\label{sec:InverseSquareMatrix}
We present in this section, two approaches for calculating the inverse of a non-singular matrix. The first approach is based on classical factorization methods such as the $LU$ and the $QR$ decomposition. These classic decomposition methods are augmented to incorporate the calculation of the inverse decomposition online. We shall indeed, modify these classic algorithms to the extent where the direct decompositions and the inverse decomposition will share the time frame.  The second approach is a completely new method. Although it is simple to understand and implement. With the help of the new method developed in this article (based on the combinatorics calculations), we shall split a given non-singular matrix into two upper and lower matrices, and proceed with the inverse calculation iteratively, while each iteration calls the fast triangular inversion. 

\subsection{Inverse matrix by augmenting classic decomposition}\label{subsec:inversedecomposition}
We present in the sequel a modified version of both $QR$ and $LU$ (Crout) algorithms to compute the inverse factorization directly and not after already having the forward decomposition. In fact, we accordingly incorporate triangular inversion algorithms such as (\textbf{CRIT} or \textbf{COMBRIT}) in these well-known matrix decomposition algorithms to produce the inverse factorization.

\subsubsection{SQR: Inverse QR factorization}
A Modified Gram-Schmidt algorithm for the forward $QR$ factorization is augmented to incorporate the inverse evaluation of the upper triangular matrix $R$. 
As described in the bellow algorithm \textbf{SQR} the evaluation of the matrix $S$ is not dependent on the complete evaluation of the upper triangular matrix $R$. Instead, as each column of $R$ is computed, the corresponding column of $S$ is simultaneously computed.
\begin{algorithm}[H]
\begin{algorithmic}[1]
\State $R_{1,1} \gets \|\v_1\|_{2}$  
\If{$R_{1,1} = 0$} 
    \State \text{stop}
\Else 
    \State $D_{1,1} \gets \dfrac{1}{R_{1,1}}$
    ,\quad $q_{1} \gets D_{1,1} \v_{1}$, \quad $S_{1,1}\gets 1$
\EndIf 

\For{$j=2$ to $n$}
    \State $\hat{\q} \gets \v_{j}$
    \For{$i=1$ to $j-1$}
        \State $R_{i,j} \gets \hat{\q}^{t}\q_{i}$
        \State $\hat{q}\gets \hat{\q}-R_{i,j}\q_{i}$
    \EndFor
    \For{$i=1$ to $j-1$}
    \State $S_{i,j} \gets - \dfrac{1}{R_{j,j}} S_{i,i:j-1}^{t} R_{i:j-1,j}$
    \EndFor
    \State $R_{j,j}=\|\hat{\q}\|_{2}$
 \If{$R_{j,j} = 0$} 
    \State \text{stop}
\Else 
    \State $D_{j,j}\gets \dfrac{1}{R_{j,j}}$
    ,\quad $\q_{j} \gets D_{j,j}\hat{\q}$ , \quad $S_{j,j}\gets 1$
\EndIf    
\EndFor 
\State $S\gets D \cdot S $
\end{algorithmic}
\caption{SQR}\label{algo:sqr}
\end{algorithm}

To decompose a matrix into an orthogonal matrix and an upper triangular matrix, we can use the $QR$ decomposition. The classical $QR$ decomposition can be achieved using the modified Gram-Schmidt algorithm. However, this algorithm is known to have numerical stability issues. It is therefore advisable to use more stable algorithms, such as Householder $QR$ or Givens $QR$ for more dedicated applications.

The \textbf{SQR} algorithm introduces an iterative approach for computing the inverse matrix $S$ (hence the inverse decomposition of the initial matrix), avoiding the need to wait for the complete evaluation of the matrix $R$. By computing $S$ column by column in parallel with the evaluation of $R$, the algorithm provides immediate results at each step. Additionally, the SQR algorithm incorporates the triangular inversion method, allowing the inverse of $R$, denoted as $S$, to be computed as the decomposition progresses. As each new column of $R$ is computed, its corresponding column in the inverse matrix $S$ is simultaneously calculated. This integration of the triangular inversion method within the \textbf{SQR} algorithm enables a dynamic and progressive computation of the inverse, enhancing the algorithm's efficiency and making it suitable for handling large matrices and various computational applications.

It is important to note that the \textbf{SQR} algorithm \ref{algo:sqr} assumes that the matrix has full column rank. If a column is zero, the algorithm will stop, and the matrix cannot be decomposed.
 
Algorithm \textbf{SQR} \ref{algo:sqr} decomposes then 
\begin{eqnarray*}
    Q R &=& A \\
    S Q^{t} &=& A^{-1}
\end{eqnarray*}
\subsubsection{SKUL: Inverse LU factorization}
We introduce an augmented algorithm based on the Crout method. This algorithm incorporates the triangular inversion technique, allowing for the simultaneous computation of the inverse decomposition alongside the direct $LU$ decomposition. Similar to the previous approach, we utilize the \textbf{CRIT} or \textbf{COMBRIT} algorithm twice within this algorithm, as we construct two triangular matrices.

The augmented $LU$ algorithm, outlined below, combines the advantages of the $LU$ decomposition and the triangular inversion technique:

\begin{algorithm}[htbp]
\begin{algorithmic}[1]
\For{$i=1$ to $n$} 
    \State $L_{i, 1} \gets A_{i, 1}$
    , $U_{i,i} \gets 1$
    , $S_{i,i}  \gets 1$
    , $K_{i,i}  \gets 1$
\EndFor

\For{$j =2$ to $n$}
    \State $U_{1, j} \gets A_{1, j} / L_{1, 1}$ 
\EndFor
\State $S_{1,2}=-U_{1,2}$
, $D_{1,1}=1/L_{1,1}$
\For{$i =2$ to $n$}
    \For{$j=2$ to $i$}
        \State $L_{i, j} = A_{i, j} - L_{i, 1:j - 1} \cdot U_{1:j - 1, j}$
    \EndFor
    \State $D_{i,i}=1/L_{i,i}$;
 
        \For{$j=1$ to $i-1$}
            \State $K_{i,j} \gets - L_{i,1:i} \cdot K_{1:i,j}  D_{i,i}$
        \EndFor
 
   \For{$j=i+1$ to $n$}
        \State $U_{i,j} \gets \left(A_{i, j} - L_{i, 1:i - 1} \cdot U_{1:i - 1, j}\right) / L_{i,i}$
    \EndFor
    \If{$i<n$}
        \For{$l=1$ to $i$}
            \State $S_{l,i+1} \gets - S_{l,l:i} \cdot U_{l:i,i+1}$
        \EndFor
    \EndIf
\EndFor
\State $K\gets K\cdot D$
\end{algorithmic}
\caption{SKUL}\label{algo:skul}
\end{algorithm}

In the \textbf{SKUL} algorithm, the calculation of the inverse matrices S and K is done in a synchronized manner with the construction of the matrices $L$ and $U$ during the $LU$ decomposition. As each column of $L$ is computed, the corresponding row of $K$ is immediately calculated using $\textbf{CRIT}$. Likewise, as each column of $U$ is computed, the corresponding row of $S$ is simultaneously evaluated using \textbf{CRIT}$^{\star}$. This simultaneous computation of the inverse matrices ensures that the inverse decomposition is formed progressively and dynamically throughout the $LU$ decomposition process. By synchronizing the calculation of rows in $K$ with the computation of columns in $L$, and the calculation of rows in $S$ with the computation of columns in $U$, the \textbf{SKUL} algorithm achieves efficient and real-time computation of the inverse decomposition

Algorithm \textbf{SKUL} \ref{algo:skul} decomposes then 
\begin{eqnarray*}
    L U &=& A \\
    S K &=& A^{-1}
\end{eqnarray*}
\subsection{Recursive Split and Inverse Method for Non-singular Matrices}\label{subsec:RSI}
\subsubsection{Element-wise RSI method}
Given a non-singular matrix $A$, we can express it as the sum of two triangular matrices, $M$ (non-singular) and $N$ (singular), such that $A = M + N$. Let's assume that $M$ is non-singular, while $N$ is a singular matrix. By appropriately permuting the rows or columns of the original matrix $A$, we can ensure the non-singularity of $M$, satisfying $A = M - N$. Starting from the definitions of $M$ and $N$, we have:
\begin{equation}\label{LUSum0} 
A=M+N=M(I+M^{-1}N).
\end{equation}
If we let $A_0=A$, $M_0=M$, and $N_0=N$, we can re-iterate as such $A_1=I+M_0^{-1}N_0$, which we decompose into a sum of singular and non-singular matrices $N_1$ and $M_1$ respectively. We have then 
\begin{equation}\label{LUSum1} A_1=I+M_0^{-1}N_0=M_1+N_1=M_1(I+M_1^{-1}N_1)\end{equation} 
By combining Eq.\eqref{LUSum0} and Eq.\eqref{LUSum1} we have 
$$A=M_0 A_1=M_0M_1(I+M_1^{-1}N_1)$$
With a straightforward bootstrapping argument we infer 
\begin{equation}\label{LUNEW}
    A = \underbrace{M_0M_1\cdots M_n}_\text{Lower triangular} \underbrace{(I+M_{n}^{-1}N_{n})}_\text{Upper triangular}
\end{equation}
Which is a re-invention of the triangular decomposition ($UL$ or $LU$), by means of recursive matrix split and product of inverse triangular matrices, taking advantage of the closedness under product of triangular matrices. Besides, as $N$ is a singular matrix with a diagonal full of zeros, the matrix product $M_i^{-1}N_i$ results in a matrix with rank $n_i-1$ where $n_i$ stands for the rank of $N$. Therefore, and by construction, we have the following Theorems.

\begin{thm}
Let $A=M+N$, where $M$ is an $n\times n$ lower (or upper) triangular matrix extracted from $A$, where $N$ is the complement upper (or lower) triangular matrix respectively. Assume that $M$ (or $N$) is invertible matrix, then 
\begin{equation}
\left\lbrace \begin{array}{lll}
    A &=& \left( \displaystyle\prod_{i=0}^{n} N_{i} \right) (I+N_{n}^{-1}M_{n}) \vspace{.3cm}\\
    M_{i+1} + N_{i+1} &=& 
    \begin{cases}
    A \text{ for } i=0 \\
    I+N_{i}^{-1}M_{i} \text{ for } i\geq 1.    
    \end{cases}
\end{array}\right.
\end{equation}
Equivalently, by symmetry under the condition of $M$ being an invertible matrix 
\begin{equation}
\left\lbrace \begin{array}{lll}
    A &=& \left( \displaystyle\prod_{i=0}^{n} M_{i} \right) (I+M_{n}^{-1}N_{n}) \vspace{.3cm}\\
    M_{i+1} + N_{i+1} &=& 
    \begin{cases}
    A \text{ for } i=0 \\
    I+M_{i}^{-1}N_{i} \text{ for } i\geq 1.    
    \end{cases}
\end{array}\right.
\end{equation}
\end{thm}
By taking the inverse of Eq.\eqref{LUNEW} we have the following results 
\begin{thm}
Let $A=M+N$, where $M$ is an $n\times n$ lower (or upper) triangular matrix extracted from $A$, where $N$ is the complement upper (or lower) triangular matrix respectively. Assume that $M$ (or $N$) is invertible matrix, then 
\begin{equation}
\left\lbrace \begin{array}{lll}
    A^{-1} &=&  (I+N_{n}^{-1}M_{n})^{-1} \left( \displaystyle\prod_{i=0}^{n} N_{i}^{-1} \right) \vspace{.3cm}\\
    M_{i+1} + N_{i+1} &=& 
    \begin{cases}
    A \text{ for } i=0 \\
    I+N_{i}^{-1}M_{i} \text{ for } i\geq 1.    
    \end{cases}
\end{array}\right.
\end{equation}
Equivalently, by symmetry under the condition of $M$ being an invertible matrix 
\begin{equation}
\left\lbrace \begin{array}{lll}
    A^{-1} &=&  (I+M_{n}^{-1}N_{n})^{-1} \left( \displaystyle\prod_{i=0}^{n} M_{i}^{-1} \right) \vspace{.3cm}\\
    M_{i+1} + N_{i+1} &=& 
    \begin{cases}
    A \text{ for } i=0 \\
    I+M_{i}^{-1}N_{i} \text{ for } i\geq 1.    
    \end{cases}
\end{array}\right.
\end{equation}
\end{thm}
\begin{proof}
    
\end{proof}

\begin{algorithm}[H]
\begin{algorithmic}[1]
\State $M + N \gets AP_{0}$\;
\State $U \gets I$\;
\State $V \gets I$\;
\For{$i=1$ to $n$}
\State $U \gets U\cdot M$\;
\State $V \gets V\cdot M^{-1}$\;
\State $M+N\gets \left(I+M^{-1}N\right)P_{i}$
\EndFor
\State $A^{-1}=L\cdot U$;
\State $A=L\cdot U$;
\end{algorithmic}
\caption{Element-wise Recursive Inversion (RSI) algorithm}\label{algo:rsi}
\end{algorithm}

We consider a given non-singular matrix $A$ of size $n\times n$, to which we apply the permutation (swap of columns) $P$, i.e. $AP$ such that the diagonal of the resulting permuted matrix $AP$ are all non-zeros. This way, the split $M+N=AP$ produces a non-singular triangular matrix $M$. This permutation will be applied every iteration to make sure that all splits lead to non-singular triangular matrix $M_{j}$, where $M_{j}+N_{j}=(I+M_{j-1}^{-1}N_{j-1})P_{j}$. The former assertion is trivial as per the assumption that $A$ is non-singular which writes as a product of two matrices $M$ and $(I+M_{j-1}^{-1}N_{j-1})$, where $M$ is non-singular by construction, hence $(I+M_{j-1}^{-1}N_{j-1})$ is also non-singular. 

Furthermore, it is worth mentioning that the singular triangular matrix $N_{j}$ has zeros on the diagonal, in addition to the closedness of the triangular matrices for the inversion operation, $M_{j}^{-1}$ is also a triangular matrix. Therefore in the case of $M$ upper triangular matrix and $N$ a lower triangular, the resulting product $M_{j}^{-1}N_{j}$ has the last column full of zeros. The shift with the identity matrix in each iteration of the Algorithm \ref{algo:rsi} makes the resulting matrix 
$I+M_{j}^{-1}N_{j}$ has the last column equal to $e_n$ (c.f. $n^\text{th}$ of the canonical basis). This process reduces the split further by rank-$1$ at each iteration, where the size of the matrix $M_{j}$ to be inverted becomes only $(n-j)\times(n-j)$:

\begin{equation}
M_{j}=
\begin{pmatrix}
    a^{(j)}_{1,1} & a^{(j)}_{1,2} & \cdots & a^{(j)}_{1,n-j} & 0 & \cdots & 0 \\
     & a^{(j)}_{2,2} & \cdots & a^{(j)}_{2,n-j} & 0 & \cdots &0 \\
     &  & \ddots & \vdots & \vdots & \ddots & \vdots \\
     &  & & a^{(j)}_{n-j,n-j} & 0 & & 0 \\
     & \bigzero  &  &  & 1 & \cdots  & 0 \\
     &  &  &  &  & \ddots & \vdots \\
     &  &  &  &  &  & 1 \\
  \end{pmatrix}
\end{equation}

Although the algorithm constructs a mathematically correct inverse of a given non-singular square matrix, it suffers from an overwhelming computational complexity overhead. In fact, its total time complexity sums up to 

\begin{eqnarray}\label{complexitySRI}
    \Invf^\textbf{SRI}(m 2^{k+1})&=& \sum_{j=0}^{m 2^{k+1}-1} \left(\Invt^\textbf{CRIT}(m 2^{k+1}-j) + \Pul(m 2^{k+1}-j) +  m 2^{k+1}-j  \right)
\end{eqnarray}

For reader convenience, we include the detailed calculation of the formula in the appendix.

It is clear that such complexity is worse than $\mathcal{O}(n^3)<\Invf(m 2^{k+1})$. Although, it is easy on the other hand to verify the correctness of the algorithm on inverting non-singular matrices. 

\subsubsection{Block-wise RSI method: BRSI}
Despite the fact that the element-wise approach has a complexity that exceeds cubic order, it forms the foundation for the block-wise approach, which significantly reduces the time complexity to a subcubic order by leveraging Strassen's method for matrix-matrix multiplication.

In the sequel, we present the \textbf{BRSI} by promoting a $\gamma$-block splitting approach. We shall assume the order of the matrix $n=m 2^{k+1}$ where $m=\gamma q$, with $2\leq \gamma \leq 2^{4}$ being an integer. 

\begin{equation}\label{qg-bounds}
\begin{cases}
    &16\cdot 2^{k} \leq n \leq m2^{k} = \gamma q 2^{k} \\
    &17 \leq q\gamma \leq 31\\
    &q=\dfrac{\left [ n 2^{-k}\right]+1}{\gamma}
    \end{cases}
\end{equation}

One remark that with regards to the setting in Eqs.\eqref{qg-bounds} ( see \cite{strassen1969gaussian}), we have $16 < m < 32$.  As we shall consider the factorization $m=q\gamma$, $m$ prime numbers can be avoided by reconsidering $n=m2^{k}=(2m)2^{k-1}$. 

Established on Eq.\eqref{qg-bounds}, we provide the following preliminary formulas that we use in the complexity analysis.
\begin{align}
q^{3}\leq \left(\frac{n 2^{-k}+1}{\gamma}\right)^{3} &= \quad \frac{1}{\gamma^{3}}\left(\frac{n^3}{8^{k}}+3\cdot \:\frac{n^2}{4^{k}}+3\cdot \frac{n}{2^{k}}+1 \right)\\
q^{2}\leq \left(\frac{n 2^{-k}+1}{\gamma}\right)^{2} &= \frac{1}{\gamma^2}\left(\frac{n^2}{4^k}+\frac{2n}{2^k}+1 \right) \\
\mathcal{Q}_{3}(\gamma):=\sum_{j=0}^{\gamma-1} (\gamma-j)^{3}&=\frac{\gamma^2\left(\gamma+1\right)^2}{4}  \\
\mathcal{Q}_{2}(\gamma):=\sum_{j=0}^{\gamma-1} (\gamma-j)^{2}&= \frac{\gamma\left(2\gamma+1\right)\left(\gamma+1\right)}{6} \\
\mathcal{Q}_{1}(\gamma) :=\sum_{j=0}^{\gamma-1} (\gamma-j)^{2}&=\frac{\gamma(\gamma+1)}{2} 
\end{align}


The Block-wise Recursive Inversion (\textbf{BRSI}) algorithm, which is an extension of Algorithm \ref{algo:rsi} with $\gamma$-blocks, is introduced in a similar manner. In the following sections, we outline the key steps involved in analyzing the time complexity of the \textbf{BRSI} algorithm.
\paragraph{Steps of \textbf{BRSI}($\gamma$)}\label{algo:brsi}
Given a non-singular square matrix $A$ and a permutation $P_0$ the \textbf{BRSI} proceeds as follows
\begin{enumerate}[label=\textbf{Step}-\arabic*]
    \item: Set $\ell=0$, and ${\bf A}_{\ell}=AP_{0}$
    \item: Split the matrix ${\bf A}_{\ell}$ of order $n=m2^{k}$ into $\gamma$ blocks of order $q2^{k}$ each.
    \item: Form $\bf{L}_{\ell}$ Low blocks triangular from $A_{\ell}$, and $\bf{U}_{0}$ upper blocks triangular from $A$, such that $$\bf{L}_{\ell}+\bf{U}_{\ell}=A_{\ell} P_{\ell}$$
    \item: \label{stepdoublerecusion}
    \begin{enumerate}[label=\textbf{Step-4.}-\arabic*]
        \item \textbf{If $\bf U$ is triangular} use \textbf{COMBRITE}($\beta$) Algorithm \ref{algo:COMBRITE} to evaluate ${\bf{U}}^{-1}_{\ell}$
        \item \textbf{Else} use  \textbf{BRSI}($\beta$) Algorithm to evaluate ${\bf{U}}^{-1}_{\ell}$
    \end{enumerate}
    \item: Set $${\bf A}_{\ell+1}={\bf I}_d + {\bf{U}}^{-1}_{0} {\bf{L}}^{}_{0}$$
    \item: Repeat \textbf{Steps}$1--5$  times $(\gamma-1)$, with $\ell=\ell+1$.
\end{enumerate}
Hence, the BSRI terminates after performing $\gamma$ calls of i) $\P$ as necessary permutation operation to handle the appropriate pivoting while extracting a non-singular matrix $\mathbf{U}_{\ell}$ before the split, ii) $\Invt^\textbf{COMBRIT}$ for block/triangular matrix inversion, and iii) $\Pul$  block matrix multiplication of Upper-block triangular and lower-block triangular matrices. Finally, after the loop, we obtain 
\begin{eqnarray*}
    A P_{0}P_{1}\cdots P_{\gamma-1}  &=& \mathbf{U}_{0}\mathbf{U}_{1}\cdots\mathbf{U}_{\gamma-1}\mathbf{L}_{\gamma-1} \\
    A P &=& \mathbf{U}\mathbf{L}
\end{eqnarray*}
where $P = P_{0}P_{1}\cdots P_{\gamma-1}$ hence $P^{-1}=P_{\gamma-1}\cdots P_{1}P_{0}$, and $\mathbf{U}=\mathbf{U}_{0}\mathbf{U}_{1}\cdots\mathbf{U}_{\gamma-1}$ and $\mathbf{L}=\mathbf{L}_{\gamma-1}$.

An important aspect to highlight is the presence of double recursion within the \textbf{BRSI} algorithm, as described in \ref{stepdoublerecusion}. This feature provides flexibility in selecting the parameters $\gamma$ and $\beta$, which respectively determine the splitting of the square matrix into $\gamma$
 blocks for inversion and recursion for handling the inverse of triangular matrices in the combinatorial-based algorithm. The simultaneous presence of these two levels of recursion in the \textbf{BRSI} algorithm offers a powerful tool for optimizing the inversion process. By carefully adjusting the values of $\gamma$ and $\beta$, researchers can fine-tune the algorithm to suit specific problem characteristics and computational requirements. This flexibility allows for tailoring the \textbf{BRSI} algorithm to achieve optimal performance and efficiency in various scenarios. 

The inversion's complexity of a full non-singular square matrix of order $n=m2^{k}$ satisfies 
\begin{equation}\label{invf0}
\begin{array}{cc}
    \Invf^\textbf{BSRI}(m 2^{k}) -\Invt^\textbf{algo}(m 2^{k})-\Pul(m2^{k}) -(\gamma-1)\Ptt(m2^{k}) \\ = \\ \sum_{j=0}^{\gamma-1} \left(\Invt^\textbf{algo}((\gamma-j) q 2^{k}) + \Pul( (\gamma-j) q 2^{k}) +  \gamma q 2^{k}  +\P((\gamma-j) q 2^{k})\right)
\end{array}    
\end{equation}
In our analysis, we will consider the worst-case complexity scenario, where the choice of algorithm, denoted as \textbf{algo}, can be either \textbf{CRIT} or \textbf{COMBRIT}. To provide a fair comparison, we will focus on the worst-case complexity that does not benefit from recursion. Specifically, when considering \textbf{CRIT}, we obtain the following expression:
\begin{align}
\sum_{j=0}^{\gamma-1} \Invt^\textbf{CRIT}((\gamma-j) q 2^{k} &\leq&  1.16\cdot  n^{\log_{2}(7)}  - 2.78\cdot n^{2}+ 461\cdot n  -  3472. \label{ubnd1}\\
\sum_{j=0}^{\gamma-1}\Pul\left((\gamma-j) q 2^{k}\right) &\leq& 1.0432 n^{\log_{2}(7)}+26.56 n^{2} + 309.67 n\log_{2}(n)  -  693.25 n \label{ubnd2}\\
\sum_{j=0}^{\gamma-1} \P((\gamma-j) q 2^{k}) &\leq&  2 n^{2} -3.32 \cdot n.\label{ubnd3}
\end{align}

The complete derivation of the upper bounds Eqs.\eqref{ubnd1}-\eqref{ubnd3} are reported in the appendix section. It has been found that the best $\gamma$-block split strategy is actually $2$-blocks split. In fact, the complexity is an increasing function of $\gamma$, which imposes taking $\gamma$ as its minimum which is $2$.  These upper bounds are, therefore, with $\gamma=2$. Furthermore, with Eq.\eqref{itm2k} and Eq.\eqref{pulm2k} the upper bound for the total complexity sums up to 
\begin{equation} 
   \Invf(m 2^{k}) \leq   4.1472n^{\log _2\left(7\right)} + 728.76n^2 - 98.807n+564.905n\log _2\left(n\right)-5765.33
\end{equation}

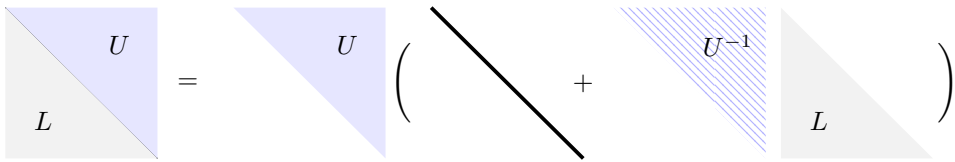
\begin{figure}[H]
\centering
 \begin{tikzpicture}[scale=2]
   \draw (0,1) -- (1,0);
  \fill[gray!10] (0,0) -- (1,0) -- (0,1) -- cycle; 
  \fill[blue!10] (1,0) -- (1,1) -- (0,1) -- cycle; 
  \fill[blue!10] (2.5,0) -- (2.5,1) -- (1.5,1) -- cycle; 
  \draw[line width=0.5mm] (2.8,1) -- (3.8,0); 
  \fill[pattern=north west lines,pattern color=blue!30] (5,0) -- (5,1) -- (4,1) -- cycle; 
  \fill[gray!10] (5.1,0) -- (6.1,0) -- (5.1,1) -- cycle; 
  \begin{scope}
  \draw (.75,.75) node {$\sc U$};
  \draw (2.25,.75) node {$\sc U$};
  \draw (4.75,.75) node {$\sc U^{-1}$};
  \draw (.25,.25) node {$\sc L$};
  \draw (5.35,.25) node {$\sc L$};
  \draw (1.2,.5) node {$=$};
  \draw (2.6,.5) node {$\Bigg($};
  \draw (3.8,.5) node {$+$};
  \draw (6.2,.5) node {$\Bigg)$};
  \end{scope}
\end{tikzpicture}
\caption{Sketch of the \textbf{RSI} \ref{algo:brsi} split decomposition method.} \label{fig:decomposition}
\end{figure}
\begin{figure}[H]
\centering
\begin{tikzpicture}[scale=0.4]
\newcommand{\size}{5}
\begin{scope} 
\foreach \i in {1,...,\size} {
    \foreach \j in {1,...,\size} {\ifnum\j>\i\draw[fill=gray!10] (\i,-\j) rectangle ++(1,-1);\fi}}
 Upper 
\foreach \i in {0,...,\size} {
\foreach \j in {0,...,\size} {\ifnum\j<\i \draw[fill=blue!10] (\i,-\j-1) rectangle ++(1,-1); \fi}}    
\end{scope}
\draw (7,-2.5) node {$=$};
\draw (20,-2.5) node {$+$};
\begin{scope} 
\foreach \i in {0,...,\size} {
    \foreach \j in {0,...,\size} {
        \ifnum\j<\i \draw[fill=blue!10] (\i+\size+2,-\j-1) rectangle ++(1,-1); \fi}}        
\draw (14,-2.5) node {$\Bigg($};  
\draw[line width=0.5mm] (15,-1) -- (15+\size,-\size-1); 
\end{scope}
\begin{scope}
\foreach \i in {0,...,\size} {
\foreach \j in {0,...,\size} {\ifnum\j<\i \draw[fill=blue!10] (\i+\size+\size+12,-\j-1) rectangle ++(1,-1); \fi}}        
\foreach \i in {1,...,\size} {
\foreach \j in {1,...,\size} {\ifnum\j>\i \draw[fill=gray!10] (\i+\size+\size+\size+12.5,-\j) rectangle ++(1,-1); \fi}}        
\draw (35,-2.5) node {$\Bigg)$};
\end{scope}
\end{tikzpicture}
\caption{Sketch of the BRSI block triangular matrix decomposition method with $\gamma$ blocks.} 
\label{fig:BRSI}
\end{figure}
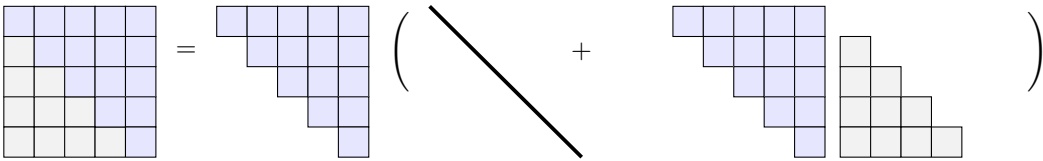

Hereafter, we describe how the algorithm \textbf{BRSI} performs for $\gamma=2$. We suppose A is a positive definite matrix,
\begin{equation*}A=
\begin{pmatrix}A_{11} & A_{12}\\ A_{21}&A_{22}\end{pmatrix}, 
\end{equation*}
with non-singular principal sub-blocks $A_{11}$ and $A_{22}$. It is worth noting here, that a permutation matrix should be applied in the case where a direct splitting doesn't lead to non-singular diagonal sub-blocks. Following the split performed in Figure~\ref{fig:BRSI} we have 

\begin{eqnarray*}A&=&
\begin{pmatrix}A_{11} & A_{12}\\0 &A_{22}\end{pmatrix} \left(
\begin{pmatrix}I_{11} & 0\\ 0 &I_{22}\end{pmatrix} + 
\begin{pmatrix}S_{11} & S_{12}\\ 0 &S_{22}\end{pmatrix}
\begin{pmatrix} 0 & 0 \\ A_{21}& 0\end{pmatrix}
\right) \\
&=& \begin{pmatrix}A_{11} & A_{12}\\0 &A_{22}\end{pmatrix}
\begin{pmatrix}I_{11} + S_{12}A_{21}& 0\\S_{22}A_{21} &I_{22}\end{pmatrix}
\end{eqnarray*}
where 
$$\begin{pmatrix}S_{11} & S_{12}\\ 0 &S_{22}\end{pmatrix}=\begin{pmatrix}A_{11} & A_{12}\\0 &A_{22}\end{pmatrix}^{-1}=\begin{pmatrix}
A_{11}^{-1} & -A_{11}^{-1}A_{12}A_{22}^{-1}\\
0 &A_{22}^{-1}
\end{pmatrix}.$$
Then computing the inverse $A^{-1}$ simply reads
\begin{eqnarray*}
A^{-1}=\begin{pmatrix}A_{11} & A_{12}\\ A_{21}&A_{22}\end{pmatrix}^{-1} 
&=&  
\begin{pmatrix}I_{11} + S_{12}A_{21}& 0\\S_{22}A_{21} &I_{22}\end{pmatrix}^{-1}
\begin{pmatrix}S_{11} & S_{12}\\ 0 &S_{22}\end{pmatrix}
\end{eqnarray*}
By further exploring the evaluation of the matrix inverse, we arrive at a formulation that is equivalent to the Schur complement. In fact,
\begin{eqnarray*}
 A^{-1}=
&=&
\begin{pmatrix}\left(I_{11} + S_{12}A_{21}\right)^{-1}& 0\\-S_{22}A_{21}\left(I_{11} + S_{12}A_{21}\right)^{-1} &I_{22}\end{pmatrix}
\begin{pmatrix}S_{11} & S_{12}\\ 0 &S_{22}\end{pmatrix}\\
&=&
\begin{pmatrix}\left(I_{11} -A_{11}^{-1}A_{12}A_{22}^{-1}A_{21}\right)^{-1}& 0\\-A_{22}^{-1}A_{21}\left(I_{11} -A_{11}^{-1}A_{12}A_{22}^{-1}A_{21}\right)^{-1} &I_{22}\end{pmatrix}
\begin{pmatrix}
A_{11}^{-1} & -A_{11}^{-1}A_{12}A_{22}^{-1}\\
0 &A_{22}^{-1}.
\end{pmatrix},
 \end{eqnarray*}
 By using the inversion lemma (also known as Woodbury identity), we can identify the result to the standard $2$-blocks $(LD)U$ identity. 
 Therefore, 
\begin{equation*}
A^{-1}=
\begin{pmatrix}\left(I_{11} -A_{11}^{-1}A_{12}A_{22}^{-1}A_{21}\right)^{-1}A_{11}^{-1}& -\left(I_{11} -A_{11}^{-1}A_{12}A_{22}^{-1}A_{21}\right)^{-1}A_{11}^{-1}A_{12}A_{22}^{-1}\\
-A_{22}^{-1}A_{21}\left(I_{11} -A_{11}^{-1}A_{12}A_{22}^{-1}A_{21}\right)^{-1}A_{11}^{-1} &
A_{22}^{-1}A_{21}\left(I_{11} -A_{11}^{-1}A_{12}A_{22}^{-1}A_{21}\right)^{-1}A_{11}^{-1}A_{12}A_{22}^{-1}+A_{22}^{-1},\end{pmatrix}
\end{equation*}
further, 
\begin{equation*}
A^{-1}=
\begin{pmatrix}\left(A_{11} -A_{12}A_{22}^{-1}A_{21}\right)^{-1}& -\left(A_{11} -A_{12}A_{22}^{-1}A_{21}\right)^{-1}A_{12}A_{22}^{-1}\\
-A_{22}^{-1}A_{21}\left(A_{11} -A_{12}A_{22}^{-1}A_{21}\right)^{-1} &
A_{22}^{-1}A_{21}\left(A_{11} -A_{12}A_{22}^{-1}A_{21}\right)^{-1}A_{12}A_{22}^{-1}+A_{22}^{-1}\end{pmatrix}
\end{equation*}

It is also worth mentioning that, if one chooses to invert $U$ instead of $L$ in the split steps, the result ends up with an $UL$-like decomposition for the initial matrix. Moreover, the splitting approach represents a generalization of the $\gamma$-block-$LU$ (or $\gamma$-block-$UL$) decompositions.  
\section{Numerical Tests}\label{sec:numerical}
The numerical simulations were conducted on an Intel Dell i7 core machine running the Ubuntu operating system (kernel version 5.15.0 SMP). The specific configuration of the machine is {\sc{x}}$86\_64$. The simulations were implemented using MATLAB version R2021b.

In this section, numerical tests of the proposed methods, namely, \textbf{SQR}, \textbf{SKUL}, and \textbf{BRSI} are presented. 

\subsection{SQR and SKUL approaches}
We present the results obtained from executing two methods: \textbf{SQR} and \textbf{SKUL}. These methods incorporate the \textbf{CRIT} method while performing the classic $QR$ and classic $LU$ (Crout) methods, respectively. We aim to provide simultaneous direct and inverse matrix decompositions.

The \textbf{SQR} method combines the classic QR method with CRIT. It utilizes the \textbf{CRIT} algorithm during the QR decomposition process to efficiently handle triangular matrices. Similarly, the \textbf{SKUL} method integrates \textbf{CRIT} within the classic $LU$ (Crout) method.

This integration allows for improved performance and accuracy in matrix decomposition tasks.
Its advantage reveals itself in the case of truncated decompositions where both direct and inverse decomposition will be provided, which forms assets for preconditioning techniques.

Let us recall that the \textbf{SQR} algorithm~\ref{algo:sqr} construct three matrices $S,Q,R$ such that 
$A=QR$ and $A^{-1}=R^{-1}Q^t=SQ^t$. Where $S$ is calculated using the newly proposed method  described in algorithm \ref{algo:sqr}.
For a given (non-singular) random matrix such as 
\begin{equation*}
 A=\left(\begin{array}{ccccc} 3.4932 & 1.6986 & 1.0719 & 1.5279 & 2.0805\\ 1.6986 & 2.9712 & 1.1746 & 1.0648 & 2.3337\\ 1.0719 & 1.1746 & 1.8540 & 0.6077 & 1.4418\\ 1.5279 & 1.0648 & 0.6077 & 2.1238 & 1.2346\\ 2.0805 & 2.3337 & 1.4418 & 1.2346 & 3.8037 \end{array}\right)   
\end{equation*}
the corresponding $Q, R$ decomposition is given by 
\begin{equation*}
    \underbrace{\left(\begin{array}{ccccc} 0.7300 & -0.5538 & -0.1454 & -0.3245 & -0.1844\\ 0.3550 & 0.7573 & -0.3205 & -0.1045 & -0.4323\\ 0.2240 & 0.1427 & 0.9303 & -0.0130 & -0.2526\\ 0.3193 & -0.0882 & -0.0688 & 0.9386 & -0.0683\\ 0.4348 & 0.3028 & 0.0771 & -0.0524 & 0.8430 \end{array}\right)}_{Q},
    \underbrace{\left(\begin{array}{ccccc} 4.7854 & 3.9123 & 2.4356 & 2.8443 & 4.7179\\ 0 & 2.0897 & 0.9435 & 0.2334 & 1.8639\\ 0 & 0 & 1.2617 & -0.0491 & 0.4990\\ 0 & 0 & 0 & 1.3136 & 0.0215\\ 0 & 0 & 0 & 0 & 1.3652 \end{array}\right)}_{R}
\end{equation*}
while the $S, Q^T$ is calculate as
\begin{equation*}
   \underbrace{\left(\begin{array}{ccccc} 0.2090 & -0.3912 & -0.1108 & -0.3871 & -0.1414\\ 0 & 0.4785 & -0.3578 & -0.0984 & -0.5210\\ 0 & 0 & 0.7926 & 0.0296 & -0.2901\\ 0 & 0 & 0 & 0.7613 & -0.0120\\ 0 & 0 & 0 & 0 & 0.7325 \end{array}\right)}_{S}, 
    \underbrace{\left(\begin{array}{ccccc} 0.7300 & 0.3550 & 0.2240 & 0.3193 & 0.4348\\ -0.5538 & 0.7573 & 0.1427 & -0.0882 & 0.3028\\ -0.1454 & -0.3205 & 0.9303 & -0.0688 & 0.0771\\ -0.3245 & -0.1045 & -0.0130 & 0.9386 & -0.0524\\ -0.1844 & -0.4323 & -0.2526 & -0.0683 & 0.8430 \end{array}\right)}_{Q^t}
\end{equation*}
On the other hand, and for the same matrix $A$ as above, the results for the $S, K, U,$ and $L$ decomposition is as follows
\begin{equation*}
    \underbrace{\left(\begin{array}{ccccc} 1 & 0.4863 & 0.3069 & 0.4374 & 0.5956\\ 0 & 1 & 0.3046 & 0.1500 & 0.6163\\ 0 & 0 & 1 & 0.0308 & 0.3022\\ 0 & 0 & 0 & 1 & 0.0810\\ 0 & 0 & 0 & 0 & 1 \end{array}\right)}_{U},
    \underbrace{\left(\begin{array}{ccccc} 3.4932 & 0 & 0 & 0 & 0\\ 1.6986 & 2.1452 & 0 & 0 & 0\\ 1.0719 & 0.6534 & 1.3261 & 0 & 0\\ 1.5279 & 0.3218 & 0.0408 & 1.4059 & 0\\ 2.0805 & 1.3221 & 0.4007 & 0.1139 & 1.6195 \end{array}\right)}_{L}
\end{equation*}
\begin{equation*}
    \underbrace{\left(\begin{array}{ccccc} 1 & -0.4863 & -0.1588 & -0.3596 & -0.2188\\ 0 & 1 & -0.3046 & -0.1407 & -0.5129\\ 0 & 0 & 1 & -0.0308 & -0.2997\\ 0 & 0 & 0 & 1 & -0.0810\\ 0 & 0 & 0 & 0 & 1 \end{array}\right)}_{S},
    \underbrace{\left(\begin{array}{ccccc} 0.2863 & 0 & 0 & 0 & 0\\ -0.2267 & 0.4662 & 0 & 0 & 0\\ -0.1197 & -0.2297 & 0.7541 & 0 & 0\\ -0.2557 & -0.1000 & -0.0219 & 0.7113 & 0\\ -0.1351 & -0.3167 & -0.1851 & -0.0500 & 0.6175 \end{array}\right)}_{K}
\end{equation*}

\begin{table}
 \begin{tabular}{ccccccc}\toprule
Matrix Size & LU & SKUL & Overhead Ratio & QR & SQR & Overhead Ratio \\\midrule
16$\times$16 & 4.22e-03 & 9.03e-03 & 2.2 & 1.09e-03 & 1.97e-03 & 1.80 \\
32$\times$32 & 2.09e-03 & 3.41e-03 & 1.63 & 1.80e-04 & 1.27e-03 & 7.08 \\
64$\times$64 & 6.13e-03 & 1.31e-02 & 2.14 & 8.86e-04 & 2.09e-03 & 2.36 \\
128$\times$128 & 2.49e-02 & 5.42e-02 & 2.17 & 4.05e-03 & 8.24e-03 & 2.04 \\
256$\times$256 & 1.18e-01 & 2.39e-01 & 2.02 & 2.47e-02 & 5.78e-02 & 2.34 \\
512$\times$512 & 5.29e-01 & 1.19 & 2.25 & 1.81e-01 & 3.11e-01 & 1.72 \\
1024$\times$1024 & 3.53 & 7.63 & 2.16 & 1.29 & 2.30 & 1.78 \\
\bottomrule
\end{tabular}
\caption{Statistic calculation of the time complexity overheads of the augmented classical $QR$ and $LU$ method.}\label{taboverheads}
\end{table}
Table \ref{taboverheads} presents the overhead results obtained from comparing the $LU$ and \textbf{SKUL} algorithms, as well as the $QR$ and \textbf{SQR} algorithms, for various matrix sizes. The overhead represents the additional time required by the modified algorithms compared to their conventional counterparts. From the results, it is evident that both the \textbf{SKUL} and \textbf{SQR} algorithms exhibit higher overhead ratios compared to the $LU$ and $QR$ algorithms, respectively. This is because incorporating the triangular inversion technique introduces additional computations, leading to longer execution times.

For the $LU$ and \textbf{SKUL} pair, the overhead ratios range from approximately 2.06 to 2.43, indicating that the \textbf{SKUL} algorithm has an overhead of around 2 times compared to the LU algorithm. Similarly, for the $QR$ and \textbf{SQR} pair, the overhead ratios range from approximately 1.46 to 2.76, indicating that the \textbf{SQR} algorithm has an overhead of around 1.5 to 2.8 times compared to the $QR$ algorithm. This indicates that the modified algorithms take roughly almost two times longer than their conventional counterparts while producing both direct and inverse decomposition. These findings highlight the trade-off between the benefits of triangular inversion incorporated in classical methods and the associated increase in computational complexity. Note that if one does first $LU$ then would apply \textbf{CRIT} to inverse $U$ and $L$ to form $S$ and $K$ respectively, would take longer wall-time computation. A report of the run-time performance of the \textbf{CRIT} method is given in the next section, when one can clearly see that the cost of making $S,K,U,L$ would take the usual runtime of obtaining the direct decomposition, with additional twice times the runtime for the evaluation of \textbf{CRIT}. This shows how important is the incorporation of the \textbf{CRIT} within the classical codes. Furthermore, if this implementation uses low-level programming such as BLAS, it would greatly decrease the ratio. 

It is worth mentioning that these overheads are mainly due to the fact that the \textbf{CRIT} doesn't benefit from any recurrent relation that decreases the handled matrix size within the iterations. On the other hand, we believe that the incorporation of \textbf{COMBRIT} method in the classical decomposition algorithms would further accelerate their computation wall-time and reduce the computational overhead. 

\subsection{Runtime performance of the combinatorial Based square matrix inversion}

Table \ref{tab:combrit} showcases a comparison of the CPU time (in seconds) for two different methods, namely Colomn Recursive Inversion Triangular \textbf{CRIT} and Combinatorial-based Block Recursive Inversion Triangular \textbf{COMBRIT}. These methods are specifically designed for non-singular triangular matrices.

The table includes various matrix sizes, ranging from $16\times16$ to $1024\times1024$, and displays the average CPU time obtained from $10$ runs of each algorithm.

Upon examining the results, we observe that the \textbf{COMBRIT} method consistently exhibits lower CPU times, especially when $\beta=2$ as it has been dictated by the theoretical complexity analysis. Furthermore, the \textbf{COMBRIT} outperforms the reference \textbf{CRIT} method, and tis is expected since \textbf{CRIT} doesn't benefit from the reduction of calculation through recursive iterations. 

However, as the matrix size grows, the \textbf{COMBRIT} method experiences a noticeable increase in CPU time, when higher values of the parameter $\beta$ are considered. This is due to the complexity overhead generated by the combinatorial approach. Nonetheless, it is recalled that these implementations are sequential while the \textbf{COMBRIT} approach doesn't take advantage of its natural parallel computing. 

However, the \textbf{COMBRIT} method maintains its efficiency and outperforms \textbf{CRIT} for larger matrices when $\beta$ is chosen as $\beta=2,4$.

\begin{table}[H]
  \centering
  \begin{tabular}{cccccc}
     - & \multicolumn{5}{c}{CPU time in second (s)}\\
    \toprule
    \multirow{2}{*}{Matrix Size} & \multirow{2}{*}{CRIT} & \multicolumn{4}{c}{COMBRIT} \\
    \cmidrule(lr){3-6} &  & $\beta=2^1$ & $\beta=2^2$ & $\beta=2^3$ & $\beta=2^4$ \\
    \midrule
 $16\times16$ & 1.32e-03 & 1.97e-03 & 8.60e-04 & 8.58e-04 & 8.32e-04\\
 $32\times32$ & 2.04e-03 & 4.76e-02 & 1.62e-02 & 1.78e-02 & 4.05\\
 $64\times64$ & 6.02e-03 & 1.02e-03 & 2.24e-03 & 1.27e-02 & 4.13\\
 $128\times128$ & 2.55e-02 & 1.52e-03 & 1.89e-03 & 1.62e-02 & 4.54\\
 $256\times256$ & 1.12e-01 & 2.70e-03 & 4.85e-03 & 2.82e-02 & 4.78\\
 $512\times512$ & 5.12e-01 & 6.69e-03 & 7.14e-03 & 4.82e-02 & 1.10e+01\\
 $1024\times1024$ & 3.33 & 4.15e-02 & 3.06e-02 & 1.21e-01 & 1.77e+01\\
    \bottomrule
  \end{tabular}
  \caption{Averaged Run-time performance in second, with statistics calculated over $10$-times runs of the algorithms.}\label{tab:combrit}
\end{table}

\begin{table}[H]
  \centering
  \begin{tabular}{ccccccc}
     - & \multicolumn{6}{c}{CPU time in second (s)}\\
    \toprule
    \multirow{2}{*}{Matrix Size} & \multirow{2}{*}{GJI} & \multirow{2}{*}{RSI} & \multicolumn{4}{c}{BRSI($\gamma=\beta$)} \\
    \cmidrule(lr){4-7} &  &  & $\beta=2^1$ & $\beta=2^2$ & $\beta=2^3$ & $\beta=2^4$ \\
    \midrule
 $16\times16$ & 1.25e-03 & 6.73e-03 & 5.51e-03 & 4.27e-03 & 3.98e-03 & 3.70e-03\\
 $32\times32$ & 1.55e-04 & 2.82e-02 & 7.05e-04 & 6.33e-04 & 5.61e-03 & 1.91\\
 $64\times64$ & 6.84e-04 & 2.60e-01 & 9.38e-04 & 1.39e-03 & 5.50e-03 & 1.90\\
 $128\times128$ & 1.73e-02 & 1.99 & 1.37e-03 & 1.17e-03 & 7.22e-03 & 2.19\\
 $256\times256$ & 2.66e-01 &  1.665e+01 & 3.58e-03 & 5.75e-03 & 1.35e-02 & 2.34\\
 $512\times512$ & 1.97 & 3.09e+02 & 1.36e-02 & 1.49e-02 & 7.45e-02& --\\
 $1024\times1024$ & 2.73e+01 & 3.92e+03 & 8.84e-02 & 9.05e-02 & 2.12e-01& --\\
    \bottomrule
  \end{tabular}
  \caption{Averaged Run-time performance in second, with statistics calculated over $10$-times runs of the algorithms. The \textbf{BRSI} \eqref{algo:brsi} algorithm runs with two blocks i.e. $\gamma=\beta$, that are also used for the inversion while calling the \textbf{COMBRIT} method \ref{algo:COMBRITE}}\label{tab:brsi}
\end{table}

Table \ref{tab:brsi} provides a comprehensive comparison of the CPU time (in seconds) for three different matrix inversion methods: Gauss Jaurgan Inversion (GJI), \textbf{RSI}, and its block version \textbf{BRSI}. The algorithms were evaluated on various matrix sizes, ranging from $16\times16$ to $1024\times1024$. The results were obtained by averaging the run-time performance over $10$ executions of each algorithm.

Upon analyzing the results, it is evident that the \textbf{RSI} method consistently demonstrates slower computation times across all matrix sizes. This is expected since as demonstrated theoretically its complexity is super-cubic. On the other hand, the \textbf{BRSI} method exhibits a notable decrease in CPU time as the matrix size grows, indicating its computational efficiency for larger matrices. However, as the parameter $\gamma=\beta$ increases, representing the size of the blocks used in the recursion in the splitting and combinatorial, the \textbf{BRSI} method achieves increasingly lower CPU times.  

For relatively smaller matrix sizes, the \textbf{BRSI} method with $\beta=2^4$ lags behind \textbf{RSI/GJI} in terms of CPU time, this is due to the complexity overheads of the combinatorics involving more matrix multiplications in its process. However, as the matrix size grows, the \textbf{BRSI} method quickly surpasses \textbf{RSI/GJI}, demonstrating its ability to handle larger-scale computations efficiently.

Notably, for the largest matrix size in the table (1024x1024), the \textbf{BRSI} method achieves a considerable improvement over \textbf{RSI/GJI}, with a significantly lower CPU time. This highlights the effectiveness of the \textbf{BRSI} approach for handling complex and computationally demanding tasks, such as large-scale matrix inversions. The results provide valuable insights into the efficiency of the proposed block method (\textbf{BRSI}) compared to traditional inversion methods. These findings make the \textbf{BRSI} method a promising approach for practical applications that require fast and accurate matrix inversions.

\section{Conclusion}\label{sec:conclude}

In this paper, we have presented novel methods for computing the inverse of non-singular triangular matrices. Our study includes the analysis of several algorithms, namely \textbf{COMBRIT}, \textbf{SQR}, \textbf{SKUL}, and \textbf{BRSI}, which provide efficient and accurate solutions for inverse factorization tasks.
The \textbf{SQR} and \textbf{SKUL} algorithms are specifically designed for the inverse decomposition of QR and LU matrices, respectively. The \textbf{COMBRIT} method utilizes combinatorial calculations based on the indexes of the entries in the initial triangular matrix, enabling a direct computation of its inverse without the need for iterative procedures. On the other hand, the \textbf{BRSI} method employs a matrix splitting approach, where the given square matrix is divided into a sum of triangular matrices. This technique takes advantage of the recurrence provided by \textbf{COMBRIT} to construct the inverse matrix iteratively.
We have conducted a comprehensive analysis of the time complexity of these algorithms, demonstrating their effectiveness across various matrix sizes. The results of numerical tests and implementations indicate that our proposed algorithms outperform traditional techniques, especially for larger matrices. Notably, the \textbf{BRSI} method exhibits promising performance when the parameter $\beta$ is appropriately chosen, making it a valuable tool for practical applications that require efficient and accurate matrix inversions.
Furthermore, our research introduces the concept of combinatorial-based approaches and recurrent techniques for triangular decomposition and inversion. These innovative methods enhance the efficiency and speed of the computations, allowing for more dynamic computation of inverse matrices.

\begin{itemize}
    \item \textbf{CRIT}: Column Recursive inverse of triangular matrices.  
    \item \textbf{COMBRIT}: Combinatorial (Block) recursive inverse of triangular matrices. 
    \item \textbf{SKUL}: Inverse factorization with augmented classical $LU$ factorization. 
    \item \textbf{SQR}: Inverse factorization with augmented classical $QR$ factorization. 
    \item \textbf{BRSI}: Inverse Factorization based on recursive, split, and block inverse of triangular matrices. 
\end{itemize}

\section*{Acknowledgment}
The author would like to acknowledge the support received through the external research grant number 8434000491 at the Emirates Nuclear Technology Center at Khalifa University.

\section*{Data and codes availability}
In the interest of transparency and reproducibility, the MATLAB codes used in this research study, including the implementation of all algorithms, are made publicly available online through the GitHub repository \url{https://github.com/riahimk/Combinatorial_Inversion.git}. This allows researchers and interested parties to access and review the codes, thereby promoting transparency and facilitating the replication of our findings. 

\section*{Appendix}
\appendix{Complexity calculation for \textbf{SRI}}

We breakdown the calculation of the complexity formula Eq.\eqref{complexitySRI} as follows:  
Using the fact that 
\begin{eqnarray}\label{prodUL00}
    \Pul(m 2^{k+1}) &=& 2 \Pul(m 2^{k}) + \Pff(m 2^{k}) + 2\Ptf(m 2^{k}) \notag\\
    &=& 2^{k+1} \Pul(m) + \sum_{j=0}^{k}2^{j}\Pff(m 2^{k-j}) + 2^{j+1}\Ptf(m 2^{k-j}) \notag\\
    &=& 2^{k+1} \Pul(m) + \sum_{j=0}^{k} 2^{j} \left( (5+2m)m^{2}7^{k-j}-6(m2^{k-j})^{2}\right)\notag\\
     &&+ \sum_{j=0}^{k}2^{j+1} \left( \left(m^{3}-\frac{13}{2}m^{2}-\dfrac{1}{3}\right) 2^{2(k-j)} +m^{2}\frac{15}{2}2^{(k-j)}
    +\frac{7^{(k-j)}}{3} \right)\notag\\
&=& 2^{k+1} \Pul(m) +  \dfrac{m^{2}(2m+5)}{5} 7^{k+1} -\dfrac{2m^{2}(m-5)}{5}2^{k+1} -3\cdot m^{2} 4^{k+1}\notag\\
&& + \left(m^{3}-\frac{13}{2}m^{2}-\dfrac{1}{3}\right)(4^{k+1}-2^{k+1})
     + (1+k)\frac{15m^{2}}{2} 2^{k+1}    \notag\\ 
      && + \dfrac{2\cdot 7^{k+1}-2^{k+2}}{15}  \notag\\
&=&   \dfrac{6m^{3}+15m^{2}+2}{15} 7^{k+1}+ \left(m^{3}-\frac{19}{2}m^{2}-\dfrac{1}{3}\right)4^{k+1} \notag\\
&&\left(\Pul(m)-\dfrac{14m^3-75km^2-160m^2-2}{10}\right) 2^{k+1} 
\end{eqnarray}
Knowing that, and by promoting the sparsity, the product of an upper triangular matrix of size $m\times m$ with a lower triangular matrix $L$ of size $m\times m$, where $L$ has diagonal entry zeros requires ${\left(m-1\right)m^2}/{3}$ multiplications and $\left({5 m^{3}-27m^{2}+46m-24}\right)/{6}$ additions. Hence, 
$$
\Pul(m)=\dfrac{7m^3-27m^2+44m-24}{6}.$$

\begin{eqnarray}\label{prodUL}
    &&\Pul(m 2^{k}) \notag\\
    &=& \left(\dfrac{6m^{3}+15m^{2}+2}{15}\right) 7^{k}+ \left(m^{3}-\frac{19}{2}m^{2}-\dfrac{1}{3}\right)4^{k} \notag\\
    &&+ \left(\dfrac{-7m^3+345m^2+220m+225km^2-114}{30}\right) 2^{k} \notag\\
    &\leq&  
    \left(\dfrac{6(\frac{n}{2^{k}}+1)^{3}+15(\frac{n}{2^{k}}+1)^{2}+2}{15}\right) 7^{k}
   +  \left((\frac{n}{2^{k}}+1)^{3}-\frac{19}{2}(\frac{n}{2^{k}}+1)^{2}-\dfrac{1}{3}\right)4^{k} \notag\\
  &&+\left(\dfrac{-7(\frac{n}{2^{k}}+1)^3+345(\frac{n}{2^{k}}+1)^2+220(\frac{n}{2^{k}}+1)+225k(\frac{n}{2^{k}}+1)^2-114}{30}\right) 2^{k} \notag\\
    &\leq&  
   \left( \frac{2}{5}\frac{n^3}{8^k}+\frac{11}{5}\frac{n^2}{4^k}+\frac{16}{5}\frac{n}{2^{k}}+\frac{23}{15} \right) 7^{k} 
   +  \left(\frac{n^3}{8^k} - \frac{13}{2}\frac{n^{2}}{ 4^k} -16\frac{n}{2^{k}}-\frac{53}{6} \right)4^{k} \notag\\
  &&+\left(-\frac{7}{30}\frac{n^3}{8^{k}}+\frac{75 k+108}{10}\frac{n^2}{4^{k}} 
  + \left( 15k+\frac{889}{30} \right) \frac{n}{2^{k}}+\frac{148+75k}{10}\right) 2^{k} \notag\\
    &=&  
   \left( \frac{2}{5}\left(\frac{8}{7}\right)^{\log_{2}(n)-k}+\frac{11}{5}\left(\frac{4}{7}\right)^{\log_{2}(n)-k}+\frac{16}{5}\left(\frac{2}{7}\right)^{\log_{2}(n)-k}+\frac{23}{15} \left(\frac{1}{7}\right)^{\log_{2}(n)-k}\right) n^{\log_{2}(7)} \notag\\
  &&+  \left(2^{\log_{2}(n)-k} - \frac{13}{2}  -16\left(\frac{1}{2}\right)^{\log_{2}(n)-k}-\frac{53}{6}\left(\frac{1}{4}\right)^{\log_{2}(n)-k} \right)n^{2} \notag\\
  &&+\left(-\frac{7}{30}4^{\log_{2}(n)-k}+\frac{75 k+108}{10} 2^{\log_{2}(n)-k}
  + \left( 15k+\frac{889}{30} \right) +\frac{148+75k}{10}\left(\frac{1}{2}\right)^{\log_{2}(n)-k}\right) n \notag\\
      &\leq&  
   \left( \frac{2}{5}\left(\frac{8}{7}\right)^5+\frac{11}{5}\left(\frac{4}{7}\right)^5+\frac{16}{5}\left(\frac{2}{7}\right)^5+\frac{23}{15} \left(\frac{1}{7}\right)^5\right) n^{\log_{2}(7)} \notag\\
  &&+  \left(2^5 - \frac{13}{2}  -16\left(\frac{1}{2}\right)^5-\frac{53}{6}\left(\frac{1}{4}\right)^5 \right)n^{2} \notag\\
  &&+\left(-\frac{7}{30}4^5+\frac{75 k+108}{10} 2^5
  + \left( 15k+\frac{889}{30} \right) +\frac{148+75k}{10}\left(\frac{1}{2}\right)^5\right) n \notag\\
     &\leq&  
   \left( \frac{33137}{36015}\right) n^{\log_{2}(7)}
  +  \left(\frac{153547}{6144} \right)n^{2} +\frac{16335}{64} n\log_{2}(n) +\left(\frac{10941}{80}\right) n  \notag\\
       &\leq&   0.921 \cdot  n^{\log_{2}(7)} +  23.8 \cdot n^{2} +255.235 \cdot n\log_{2}(n) +136.763 \cdot n \label{pulm2k}
\end{eqnarray}

We note that this formula uses sub-blocks of size $m 2^{k}$ from the initial order $m 2^{k+1}$. Despite this fact, the matrix multiplication operation still works for the matrices of order $m2^{k+1}-j$ by the simple fact that we can adjust the size of the new matrix accordingly by adding necessary columns formed by the canonical basis. We  have the following upper bound for the time complexity
\begin{eqnarray*}
    \Invf^\textbf{SRI}(m 2^{k+1})&\leq& \sum_{j=0}^{m 2^{k+1}-1} \left(\Invt^\textbf{CRIT}(m 2^{k+1}-j) + \Pul(m 2^{k+1}) +  m 2^{k+1}-j  \right)\\
    &=& \sum_{j=0}^{m 2^{k+1}-1} \left(\dfrac{2}{15} 7^{k+1}  +  m 2^{k+1}-j  \right)\\
    && +\sum_{j=0}^{m 2^{k+1}-1} 2 \Pul(m 2^{k}) + 2\Pff(m 2^{k}) + \Ptf(m 2^{k})\\
    &=& \sum_{j=0}^{m 2^{k+1}-1} \left(\dfrac{6m^{3}+15m^{2}+4}{15} 7^{k+1} +\left(m^{3}-\frac{19}{2}m^{2}-\dfrac{1}{3}\right)4^{k+1} \right)\\
    && +\sum_{j=0}^{m 2^{k+1}-1} \left( \left(m+\Pul(m)-\dfrac{14m^3-75km^2-160m^2-2}{10}\right) 2^{k+1} -j \right)
\end{eqnarray*}

\appendix{Complexity calculation for formula Eq.\eqref{ubnd1}} 

\begin{eqnarray*}
&&\sum_{j=0}^{\gamma-1} \Invt^\textbf{CRIT}((\gamma-j) q 2^{k})   \\
&=&\sum_{j=0}^{\gamma-1} \dfrac{4m^2 (5 + 2 m)}{18}\:7^{k} -  \dfrac{2m^3 + k 15m^2}{18}\: 4^{k}  
     +\dfrac{m^3-3m^2+8m-3}{3}\:2^{k}-\dfrac{m^3+2m^2 (5 + 2 m)}{9}.\\
    &=& \sum_{j=0}^{\gamma-1} \left( \frac{4m^3}{9}+\frac{10m^2}{9} \right) \:7^{k} 
     -\sum_{j=0}^{\gamma-1}\left(\frac{m^3}{9}+\frac{k^{15}m^2}{18}\right)\: 4^{k} 
     +\sum_{j=0}^{\gamma-1}\left(\frac{m^3}{3}-m^2+\frac{8m}{3}-1\right)\:2^{k} 
     -\sum_{j=0}^{\gamma-1} \left( \frac{5m^3}{9}+\frac{10m^2}{9} \right).\\
    &=& \sum_{j=0}^{\gamma-1} \left( \frac{4}{9}q^{3}(\gamma-j)^{3}+\frac{10}{9}q^{2}(\gamma-j)^{2} \right) \:7^{k} 
     -\sum_{j=0}^{\gamma-1}\left(\frac{q^{3}}{9}(\gamma-j)^{3}+\frac{k^{15}}{18}q^{2}(\gamma-j)^{2}\right)\: 4^{k} \\
     &&+\sum_{j=0}^{\gamma-1}\left(\frac{q^{3}}{3}(\gamma-j)^{3}-q^{2}(\gamma-j)^{2}+\frac{8q}{3}(\gamma-j)-1\right)\:2^{k} 
     -\sum_{j=0}^{\gamma-1} \left( \frac{5q^{3}}{9}(\gamma-j)^{3}+\frac{10q^2}{9}(\gamma-j)^{2} \right).\\   
    &=&   \left( \frac{4}{9}q^{3}\mathcal{Q}_{3}(\gamma)+\frac{10}{9}q^{2}\mathcal{Q}_{2}(\gamma) \right) \:7^{k} 
     - \left(\frac{q^{3}}{9}\mathcal{Q}_{3}(\gamma)+\frac{k15}{18}q^{2}\mathcal{Q}_{2}(\gamma)\right)\: 4^{k} \\
     &&+ \left(\frac{q^{3}}{3}\mathcal{Q}_{3}(\gamma)-q^{2}\mathcal{Q}_{2}(\gamma)+\frac{8q}{3}\mathcal{Q}_{1}(\gamma)-\gamma\right)\:2^{k} 
     -  \left( \frac{5q^{3}}{9}\mathcal{Q}_{3}(\gamma)+\frac{10q^2}{9}\mathcal{Q}_{2}(\gamma) \right).\\
&\leq&   \left( \frac{1}{9}\frac{\left(\gamma+1\right)^2}{\gamma} \left(\frac{n^3}{8^{k}}+3\cdot \:\frac{n^2}{4^{k}}+3\cdot \frac{n}{2^{k}}+1 \right)+\frac{5}{27}\frac{\left(2\gamma+1\right)\left(\gamma+1\right)}{\gamma} \left(\frac{n^2}{4^k}+\frac{2n}{2^k}+1 \right) \right) \:7^{k} \\
     &&- \left(\frac{1}{36}\frac{\left(\gamma+1\right)^2}{\gamma} \left(\frac{n^3}{8^{k}}+3\cdot \:\frac{n^2}{4^{k}}+3\cdot \frac{n}{2^{k}}+1 \right)+\frac{k 5}{36}\frac{\left(2\gamma+1\right)\left(\gamma+1\right)}{6\gamma} \left(\frac{n^2}{4^k}+\frac{2n}{2^k}+1 \right)\right)\: 4^{k} \\
     &&+ \left(\frac{1}{12}\frac{\left(\gamma+1\right)^2}{\gamma} \left(\frac{n^3}{8^{k}}+3\cdot \:\frac{n^2}{4^{k}}+3\cdot \frac{n}{2^{k}}+1 \right)-\frac{\left(2\gamma+1\right)\left(\gamma+1\right)}{6\gamma} \left(\frac{n^2}{4^k}+\frac{2n}{2^k}+1 \right)+\frac{1-2\gamma}{3}\right)\:2^{k} \\
     &&-  \left( \frac{5}{36}\frac{\left(\gamma+1\right)^2}{\gamma} \left(\frac{n^3}{8^{k}}+3\cdot \:\frac{n^2}{4^{k}}+3\cdot \frac{n}{2^{k}}+1 \right)+\frac{5}{27}\frac{\left(2\gamma+1\right)\left(\gamma+1\right)}{\gamma} \left(\frac{n^2}{4^k}+\frac{2n}{2^k}+1 \right) \right).
\end{eqnarray*}     
\begin{eqnarray*}
&\leq&   \left( \frac{1}{9}\frac{\left(\gamma+1\right)^2}{\gamma} \left(\left(\frac{8}{7}\right)^{\log_{2}(n)-k}+3\cdot \:\left(\frac{4}{7}\right)^{\log_{2}(n)-k}+3\cdot \left(\frac{2}{7}\right)^{\log_{2}(n)-k}+\left(\frac{1}{7}\right)^{\log_{2}(n)-k} \right) \right) \:n^{\log_{2}(7)} \\
&& +\left( \frac{5}{27}\frac{\left(2\gamma+1\right)\left(\gamma+1\right)}{\gamma} \left(\left(\frac{4}{7}\right)^{\log_{2}(n)-k}+2\left(\frac{2}{7}\right)^{\log_{2}(n)-k}+\left(\frac{1}{7}\right)^{\log_{2}(n)-k} \right) \right) \:n^{\log_{2}(7)} \\
     &&- \left(\frac{1}{36}\frac{\left(\gamma+1\right)^2}{\gamma} \left(2^{\log_{2}(n)-k}+3+3\left(\frac{1}{2}\right)^{\log_{2}(n)-k}+\left(\frac{1}{4}\right)^{\log_{2}(n)-k} \right) \right)\: n^{2} \\
     && - \left( \frac{k 5}{36}\frac{\left(2\gamma+1\right)\left(\gamma+1\right)}{6\gamma} \left(1+2\left(\frac{1}{2}\right)^{\log_{2}(n)-k}+\left(\frac{1}{4}\right)^{\log_{2}(n)-k} \right)\right)\: n^{2} \\
     &&+ \left(\frac{1}{12}\frac{\left(\gamma+1\right)^2}{\gamma} \left(4^{\log_{2}(n)-k} \right)+3\cdot 2^{\log_{2}(n)-k}  +3+ \left(\frac{1}{2}\right)^{\log_{2}(n)-k}   \right)\: n \\
     &&-\left(\frac{\left(2\gamma+1\right)\left(\gamma+1\right)}{6\gamma} \left(2^{\log_{2}(n)-k} +1+\left(\frac{1}{2}\right)^{\log_{2}(n)-k}  \right)+\frac{1-2\gamma}{3}\right)\: n \\
     &&-  \frac{5}{36}\frac{\left(\gamma+1\right)^2}{\gamma} \left(8^{\log_{2}(n)-k}+3 \cdot 4^{\log_{2}(n)-k}+3\cdot 2^{\log_{2}(n)-k}+1 \right)  \\
     &&-\frac{5}{27}\frac{\left(2\gamma+1\right)\left(\gamma+1\right)}{\gamma} \left(4^{\log_{2}(n)-k}+2\cdot 2^{\log_{2}(n)-k}+1 \right)\\ 
&\leq&   \left( \frac{1}{9}\frac{\left(\gamma+1\right)^2}{\gamma} \left(\left(\frac{8}{7}\right)^5+3\cdot \:\left(\frac{4}{7}\right)^5+3\cdot \left(\frac{2}{7}\right)^5+\left(\frac{1}{7}\right)^5 \right) \right) \:n^{\log_{2}(7)} \\
&& +\left( \frac{5}{27}\frac{\left(2\gamma+1\right)\left(\gamma+1\right)}{\gamma} \left(\left(\frac{4}{7}\right)^5+2\left(\frac{2}{7}\right)^5+\left(\frac{1}{7}\right)^5 \right) \right) \:n^{\log_{2}(7)} \\
     &&- \left(\frac{1}{36}\frac{\left(\gamma+1\right)^2}{\gamma} \left(2^4+3+3\left(\frac{1}{2}\right)^4+\left(\frac{1}{4}\right)^4 \right) \right)\: n^{2} \\
     && - \left( \frac{10}{36}\frac{\left(2\gamma+1\right)\left(\gamma+1\right)}{6\gamma} \left(1+2\left(\frac{1}{2}\right)^4+\left(\frac{1}{4}\right)^4 \right)\right)\: n^{2} \\
     &&+ \left(\frac{1}{12}\frac{\left(\gamma+1\right)^2}{\gamma} \left(4^5 \right)+3\cdot 2^5  +3+ \left(\frac{1}{2}\right)^5   \right)\: n \\
     &&-\left(\frac{\left(2\gamma+1\right)\left(\gamma+1\right)}{6\gamma} \left(2^4 +1+\left(\frac{1}{2}\right)^4  \right)+\frac{1-2\gamma}{3}\right)\: n \\
     &&-  \frac{5}{36}\frac{\left(\gamma+1\right)^2}{\gamma} \left(8^4+3 \cdot 4^4+3\cdot 2^4+1 \right)  \\
     &&-\frac{5}{27}\frac{\left(2\gamma+1\right)\left(\gamma+1\right)}{\gamma} \left(4^4+2\cdot 2^4+1 \right)
\end{eqnarray*}

Finally, 
\begin{eqnarray*}
\sum_{j=0}^{\gamma-1} \Invt^\textbf{CRIT}((\gamma-j) q 2^{k} &\leq&   \left(\frac{121\left(109\gamma+104\right)\left(\gamma+1\right)}{50421\gamma}\right) \:n^{\log_{2}(7)}  
     - \left( \frac{289\left(61\gamma+56\right)\left(\gamma+1\right)}{27648\gamma} \right)\: n^{2} \\
     &&+ \left( \frac{3169}{32}+\frac{7582\gamma^2+15597\gamma+7919}{96\gamma}\right)\: n  
      -  \frac{1445\left(59\gamma+55\right)\left(\gamma+1\right)}{108\gamma}.
\end{eqnarray*}

\begin{equation}
 \sum_{j=0}^{\gamma-1} \Invt^\textbf{CRIT}((\gamma-j) q 2^{k}\leq   1.16\cdot  n^{\log_{2}(7)}  - 2.78\cdot n^{2}+ 461\cdot n  -  3472.      
\end{equation}

\appendix{Complexity calculation for formula Eq.\eqref{ubnd2}} 

\begin{eqnarray*}
 &&\sum_{j=0}^{\gamma-1}\Pul\left((\gamma-j) q 2^{k}\right) \\
 &=& \sum_{j=0}^{\gamma-1} \left(\dfrac{6(\gamma-j)^{3}q^{3}+15(\gamma-j)^{2}q^{2}+2}{15}\right) 7^{k}+ \left((\gamma-j)^{3}q^{3}-\frac{19}{2}(\gamma-j)^{2}q^{2}-\dfrac{1}{3}\right)4^{k} \notag\\
    &&+ \left(\dfrac{-7(\gamma-j)^{3}q^3+345(\gamma-j)^{2}q^2+220(\gamma-j)q+225k(\gamma-j)^{2}q^2-114}{30}\right) 2^{k} \\
&=& \sum_{j=0}^{\gamma-1} \left(\frac{2}{5}q^{3}(\gamma-j)^{3}+q^{2}(\gamma-j)^{2}+\frac{2}{15}\right) 7^{k}+ \left((\gamma-j)^{3}q^{3}-\frac{19}{2}(\gamma-j)^{2}q^{2}-\dfrac{1}{3}\right)4^{k} \notag\\
    &&+ \left(\dfrac{-7}{30}q^{3}(\gamma-j)^{3}+\dfrac{345}{30}q^{2}(\gamma-j)^{2}+\dfrac{220}{30}q (\gamma-j)+\dfrac{225k}{30}q^{2}(\gamma-j)^{2}-\dfrac{114}{30}\right) 2^{k} \\
&=&  \left(\frac{2}{5}\mathcal{Q}_{3}(\gamma) q^{3}+\mathcal{Q}_{2}(\gamma)q^{2}+\frac{2}{15}\gamma\right) 7^{k}+ \left( \mathcal{Q}_{3}(\gamma) q^{3}-\frac{19}{2}\mathcal{Q}_{2}(\gamma)q^{2}-\dfrac{1}{3}\gamma\right)4^{k} \notag\\
    &&+ \left(\dfrac{-7}{30}\mathcal{Q}_{3}(\gamma)q^{3}+\dfrac{345}{30}\mathcal{Q}_{2}(\gamma)q^{2}+\dfrac{220}{30}\mathcal{Q}_{1}(\gamma) q+\dfrac{225k}{30}\mathcal{Q}_{2}(\gamma)q^{2}-\dfrac{114}{30}\gamma\right) 2^{k} \\  
&=&  \left(\frac{2}{5}\mathcal{Q}_{3}(\gamma) q^{3}+\mathcal{Q}_{2}(\gamma)q^{2}+\frac{2}{15}\gamma\right) 7^{k}+ \left( \mathcal{Q}_{3}(\gamma) q^{3}-\frac{19}{2}\mathcal{Q}_{2}(\gamma)q^{2}-\dfrac{1}{3}\gamma\right)4^{k} \notag\\
    &&+ \left(\dfrac{-7}{30}\mathcal{Q}_{3}(\gamma)q^{3}+\dfrac{345}{30}\mathcal{Q}_{2}(\gamma)q^{2}+\dfrac{220}{30}\mathcal{Q}_{1}(\gamma) q-\dfrac{114}{30}\gamma\right) 2^{k} 
    +\left(\dfrac{225k}{30}\mathcal{Q}_{2}(\gamma)q^{2} \right)\: 2^{k}\\
&\leq&  
   \frac{\left(\gamma+1\right)^2}{10\gamma} \left(\frac{n^3}{8^{k}}+3\cdot \:\frac{n^2}{4^{k}}+3\cdot \frac{n}{2^{k}}+1\right)\: 7^{k} +\left(\frac{\left(2\gamma+1\right)\left(\gamma+1\right)}{6\gamma} \left(\frac{n^2}{4^k}+\frac{2n}{2^k}+1 \right) + \frac{2\gamma}{15} \right) \: 7^{k} \\ 
 &&+ \left(  \frac{\left(\gamma+1\right)^2}{4\gamma}\left(\frac{n^3}{8^{k}}+3\cdot \:\frac{n^2}{4^{k}}+3\cdot \frac{n}{2^{k}}+1 \right)
 -\frac{19}{12} \frac{\left(2\gamma+1\right)\left(\gamma+1\right)}{\gamma}\left(\frac{n^4}{4^k}+\frac{2n}{2^k}+1 \right)
 -\frac{\gamma}{3} \right)\: 2^{k}\\
 &&+ \left(\dfrac{-7}{30}\frac{\left(\gamma+1\right)^2}{\gamma} \left(\frac{n^3}{8^{k}}+3\cdot \:\frac{n^2}{4^{k}}+3\cdot \frac{n}{2^{k}}+1 \right) 
+\dfrac{345}{180} \frac{\left(2\gamma+1\right)\left(\gamma+1\right)}{\gamma} \left(\frac{n^2}{4^k}+\frac{2n}{2^k}+1 \right)\right) 2^{k} \\  
&&+ \left( \dfrac{220}{30} \frac{\gamma+1}{2} (\frac{n}{2^{k}}+1) -\dfrac{114}{30}\gamma \right) 2^{k} + \left(\dfrac{225k}{180}\frac{\left(2\gamma+1\right)\left(\gamma+1\right)}{\gamma}\left(\frac{n^2}{4^k}+\frac{2n}{2^k}+1 \right) \right) 2^{k}\\
&=&  
\frac{\left(\gamma+1\right)^2}{10\gamma} \left( \left(\frac{8}{7}\right)^{\log_{2}(n)-k}+3\cdot \left(\frac{4}{7}\right)^{\log_{2}(n)-k}+3\cdot\left(\frac{2}{7}\right)^{\log_{2}(n)-k}+\left(\frac{1}{7}\right)^{\log_{2}(n)-k}\right)\: n^{\log_{2}(7)} \\ 
&&+\frac{\left(2\gamma+1\right)\left(\gamma+1\right)}{6\gamma} \left(
\left(\frac{4}{7}\right)^{\log_{2}(n)-k}
     +2\left(\frac{2}{7}\right)^{\log_{2}(n)-k}
     + \left(\frac{1}{7}\right)^{\log_{2}(n)-k} \right) \: n^{\log_{2}(7)} 
+\left( \frac{2\gamma}{15} \left(\frac{1}{7}\right)^{\log_{2}(n)-k} \right) \: n^{\log_{2}(7)}  \\
&&+ \left(  \frac{\left(\gamma+1\right)^2}{4\gamma}\left(2^{\log_{2}(n)-k}+3\cdot \left(\frac{4}{4}\right)^{\log_{2}(n)-k}+3\cdot \left(\frac{1}{2}\right)^{\log_{2}(n)-k}+\left(\frac{1}{4}\right)^{\log_{2}(n)-k} \right) \right) n^{2} \\
 &&-\left( \frac{19}{12} \frac{\left(2\gamma+1\right)\left(\gamma+1\right)}{\gamma}\left(\left(\frac{1}{1}\right)^{\log_{2}(n)-k}+\left(\frac{1}{2}\right)^{\log_{2}(n)-k}+\left(\frac{1}{4}\right)^{\log_{2}(n)-k} \right)
 -\frac{\gamma}{3} \right)\: n^{2}\\
 &&+ \left(\dfrac{-7}{30}\frac{\left(\gamma+1\right)^2}{\gamma} \left(4^{\log_{2}(n)-k}+3\cdot 2^{\log_{2}(n)-k}+3 +\left(\frac{1}{2}\right)^{\log_{2}(n)-k} \right) \right)\:n\\ 
 &&+\left(\dfrac{345}{180} \frac{\left(2\gamma+1\right)\left(\gamma+1\right)}{\gamma} \left(2^{\log_{2}(n)-k}+1+\left(\frac{1}{2}\right)^{\log_{2}(n)-k} \right)\right) n \\  
&&+ \left( \dfrac{220}{30} \frac{\gamma+1}{2} (1+\left(\frac{1}{2}\right)^{\log_{2}(n)-k}) -\dfrac{114}{30}\gamma \left(\frac{1}{2}\right)^{\log_{2}(n)-k}\right) n \\
&&+ \left(\dfrac{225}{180}\frac{\left(2\gamma+1\right)\left(\gamma+1\right)}{\gamma}\left(2^{\log_{2}(n)-k}+1+\left(\frac{1}{2}\right)^{\log_{2}(n)-k} \right) \right) n \log_{2}(n)
\end{eqnarray*}

\begin{eqnarray*}
&\leq&  
\frac{\left(\gamma+1\right)^2}{10\gamma} \left( \left(\frac{8}{7}\right)^{5}+3\cdot \left(\frac{4}{7}\right)^{5}+3\cdot\left(\frac{2}{7}\right)^{5}+\left(\frac{1}{7}\right)^{5}\right)\: n^{\log_{2}(7)} \\ 
&&+\left( \frac{\left(2\gamma+1\right)\left(\gamma+1\right)}{6\gamma} \left(
\left(\frac{4}{7}\right)^{5}
     +2\left(\frac{2}{7}\right)^{5}
     + \left(\frac{1}{7}\right)^{5} \right)  
+ \frac{2\gamma}{15} \left(\frac{1}{7}\right)^{5} \right) \: n^{\log_{2}(7)} \\ 
&&+ \left(  \frac{\left(\gamma+1\right)^2}{4\gamma}\left(2^{5}+3\cdot \left(\frac{4}{4}\right)^{5}+3\cdot \left(\frac{1}{2}\right)^{5}+\left(\frac{1}{4}\right)^{5} \right) \right) n^{2} \\
 &&-\left( \frac{19}{12} \frac{\left(2\gamma+1\right)\left(\gamma+1\right)}{\gamma}\left(\left(\frac{1}{1}\right)^{5}+\left(\frac{1}{2}\right)^{5}+\left(\frac{1}{4}\right)^{5} \right)
 +\frac{\gamma}{3} \right)\: n^{2}\\
 &&+ \left(\dfrac{-7}{30}\frac{\left(\gamma+1\right)^2}{\gamma} \left(4^{5}+3\cdot 2^{5}+3 +\left(\frac{1}{2}\right)^{5} \right) \right)\:n\\ 
 &&+\left(\dfrac{345}{180} \frac{\left(2\gamma+1\right)\left(\gamma+1\right)}{\gamma} \left(2^{5}+1+\left(\frac{1}{2}\right)^{5} \right)\right) n \\  
&&+ \left( \dfrac{220}{30} \frac{\gamma+1}{2} (1+\left(\frac{1}{2}\right)^{5}) -\dfrac{114}{30}\gamma \left(\frac{1}{2}\right)^{5}\right) n \\
&&+ \left(\dfrac{225}{180}\frac{\left(2\gamma+1\right)\left(\gamma+1\right)}{\gamma}\left(2^{5}+1+\left(\frac{1}{2}\right)^{5} \right) \right) n \log_{2}(n)
\end{eqnarray*}
Finally, we have 
\begin{eqnarray*}
    \sum_{j=0}^{\gamma-1}\Pul\left((\gamma-j) q 2^{k}\right) &\leq&  
\underbrace{\left( 
\frac{\left(\gamma+1\right)^2}{10\gamma} \frac{35937}{16807}
+\frac{\left(2\gamma+1\right)\left(\gamma+1\right)}{6\gamma}
\frac{1089}{16807}+ \frac{2\gamma}{252105}
\right)}_{\theta_\text{a}(\gamma)} \: n^{\log_{2}(7)} \\ 
&&+ \underbrace{\left(  \frac{\left(\gamma+1\right)^2}{4\gamma}\frac{35937}{1024}  
 -  \frac{19}{12} \frac{\left(2\gamma+1\right)\left(\gamma+1\right)}{\gamma}\frac{1057}{1024}
 -\frac{\gamma}{3} \right)}_{\theta_\text{b}(\gamma)}\: n^{2} \\
 &&+ \underbrace{\left(\frac{\left(2\gamma+1\right)\left(\gamma+1\right)}{\gamma}\dfrac{225}{180}\frac{1057}{32} \right)}_{\theta_\text{c}(\gamma)} n \log_{2}(n) \\
 && + \underbrace{\left(\frac{-260008\gamma^2-641571\gamma-381563}{1920\gamma} +  \frac{293\gamma}{80}+\frac{121}{32}\right)}_{\theta_\text{d}(\gamma)} n \\
&&= \theta_\text{a}(\gamma) n^{\log_{2}(7)}+\theta_\text{b}(\gamma) n^{2} + \theta_\text{c}(\gamma) n\log_{2}(n) + \theta_\text{d}(\gamma) n
\end{eqnarray*}
Hence, 
\begin{equation*}
    \sum_{j=0}^{\gamma-1}\Pul\left((\gamma-j) q 2^{k}\right) \leq  \theta_\text{a}(\gamma) n^{\log_{2}(7)}+\theta_\text{b}(\gamma) n^{2} + \theta_\text{c}(\gamma) n\log_{2}(n) + \theta_\text{d}(\gamma) n
\end{equation*}
Besides, for a matrix of order $n$ a Permutation involves $n(n-1)/2$ comparisons and $n(n-1)$ columns interchanges.

\appendix{Complexity calculation for formula Eq.\eqref{ubnd3}} 

\begin{eqnarray*}
\sum_{j=0}^{\gamma-1} \P((\gamma-j) q 2^{k}) &=& \frac{3}{2} \sum_{j=0}^{\gamma-1} ((\gamma-j) q 2^{k})^{2} - (\gamma-j) q 2^{k}\\
&=&  \frac{3}{2} \sum_{j=0}^{\gamma-1} (\gamma-j)^{2} q^{2} 4^{k} - (\gamma-j) q 2^{k}\\
&=&  \frac{3}{2}  \mathcal{Q}_{2}(\gamma) q^{2} 4^{k} - \mathcal{Q}_{1}(\gamma) q 2^{k}\\
&\leq&  \frac{3}{12}  \frac{\left(2\gamma+1\right)\left(\gamma+1\right)}{ \gamma} \left(\frac{n^2}{4^k}+\frac{2n}{2^k}+1 \right)  4^{k} - \frac{3}{4}(\gamma+1) (\frac{n}{2^{k}}+1) 2^{k}\\
&=&  \frac{3}{12}  \frac{\left(2\gamma+1\right)\left(\gamma+1\right)}{ \gamma} \left( 1 + 2\left(\frac{1}{2}\right)^{\log_{2}(n)-k}+\left(\frac{1}{4}\right)^{\log_{2}(n)-k} \right)  n^{2} \\&&- \frac{3}{4}(\gamma+1) \left(1+\left(\frac{1}{2}\right)^{\log_{2}(n)-k}\right) n\\
&\leq&  \frac{1089}{4096} \frac{\left(2\gamma+1\right)\left(\gamma+1\right)}{ \gamma}   n^{2} -\frac{99}{128}(\gamma+1)  n\\
&\leq&  2 n^{2} -3.32 \cdot n
\end{eqnarray*}

\bibliographystyle{plain}
\bibliography{bib_InverseDecomp.bib}
\end{document}